\let\ORIlabel\label
\let\ORIrefstepcounter\refstepcounter
\AddToHook{package/hyperref/before}{%
  \let\label\ORIlabel
  \let\refstepcounter\ORIrefstepcounter
}

\documentclass[final,onefignum,onetabnum]{siamart171218} 

\usepackage{amsmath,amssymb,amsfonts,mathtools,bm}
\usepackage{graphicx}
\usepackage{booktabs}
\usepackage{algorithm}
\usepackage[noend]{algpseudocode}
\usepackage{listings}

\newcommand{\R}{\mathbb{R}}
\newcommand{\X}{\mathcal{X}}
\newcommand{\Y}{\mathcal{Y}}
\newcommand{\Xhat}{\widehat{\mathcal{X}}}
\newcommand{\G}{\mathcal{G}}
\newcommand{\A}{\mathbf{A}}
\newcommand{\od}{\odot}
\newcommand{\eps}{\varepsilon}

\renewcommand{\emph}[1]{#1}

\newsiamthm{assumption}{Assumption}
\newsiamremark{remark}{Remark}

\headers{V. Leplat}{Joint MM for Nonnegative CP and Tucker}

\title{Joint Majorization-Minimization for Nonnegative CP and Tucker Decompositions under \texorpdfstring{$\beta$}{beta}-Divergences: Unfolding-Free Updates}
\author{Valentin Leplat}

\ifpdf
\hypersetup{
  pdftitle={Joint Majorization-Minimization for Nonnegative CP and Tucker Decompositions under beta-Divergences: Unfolding-Free Updates},
  pdfauthor={Valentin Leplat}
}
\fi

\begin{document}

\maketitle
\begin{abstract}
We study majorization-minimization methods for nonnegative tensor decompositions under the $\beta$-divergence family, focusing on nonnegative CP and Tucker models. Our aim is to avoid explicit mode unfoldings and large auxiliary matrices by deriving separable surrogates whose multiplicative updates can be implemented using only tensor contractions (einsum-style operations). We present both classical block-MM updates in contraction-only form and a joint majorization strategy, inspired by joint MM for matrix $\beta$-NMF, that reuses cached reference quantities across inexpensive inner updates. We prove tightness of the proposed majorizers, establish monotonic decrease of the objective, and show convergence of the sequence of objective values. For block-MM, we discuss how BSUM theory applies to the analysis of stationary accumulation points. For J-CoMM, we further establish, under a set of standard regularity assumptions and for one inner sweep per outer iteration, convergence of the iterates to a critical point through a KL-based analysis. Finally, experiments on synthetic tensors and the Uber spatiotemporal count tensor demonstrate substantial speedups over unfolding-based baselines and competitive runtime relative to a recent einsum-factorization framework.
\end{abstract}



\section{Introduction}

Nonnegative matrix factorization (NMF) is widely used to learn interpretable representations from nonnegative data~\cite{leeseung1999}.
Among the most effective approaches for computing an NMF of a nonnegative input matrix are block-coordinate majorization-minimization (MM) schemes, particularly for nonquadratic discrepancy measures such as the $\beta$-divergence family.
These methods produce multiplicative updates with monotonic descent and have become a standard tool in the area.

For tensor data, nonnegative CP and Tucker decompositions are natural generalizations.
However, many optimization methods rely on mode unfoldings, Khatri-Rao/Kronecker products, and large intermediate matrices, which can be costly to form and move in memory.
This motivates \emph{unfolding-free} update rules that operate directly on tensors through contractions.

We develop MM algorithms for nonnegative CP and Tucker decompositions under entry-wise $\beta$-divergences
($\beta\in[0,2)$) in an \emph{unfolding-free} form: all update numerators and denominators are written as tensor
contractions and can be implemented directly with einsum-style primitives, without explicit matricizations.
Beyond expressing classical block surrogates in contraction form, we also design a \emph{joint} majorizer tailored to
multilinear tensor models, so that several inexpensive inner block updates can be performed while reusing cached
reference quantities, rather than rebuilding a matrix-type surrogate after unfolding.

\paragraph{Contributions.}
\begin{itemize}
\item \textbf{Unfolding-free block-MM updates for CP and Tucker.}
We derive the classical MM multiplicative updates for CP and Tucker in a \emph{contraction-only} form, i.e., with
numerators/denominators written as explicit tensor contractions rather than unfolding-based kernels, and we provide
einsum recipes that avoid large auxiliary matrices.

\item \textbf{Joint majorization with cheap inner updates.}
Inspired by joint MM for matrix $\beta$-NMF, we construct a \emph{single} surrogate at a reference iterate and decrease
it through a short inner loop of inexpensive multiplicative block updates.
The key mechanism is to reuse cached reference-powered quantities required by the majorizer, reducing repeated
recomputation and memory traffic in large-scale CP/Tucker.

\item \textbf{Descent, objective-value convergence, and iterate convergence under assumptions.}
We prove tightness of the proposed majorizers and establish monotonic decrease
(per block for block-MM, and per outer iteration for joint-MM).
Consequently, the sequence of objective values converges.
For block-MM, we discuss how BSUM theory can be used to analyze stationary accumulation points under standard
regularity assumptions.
For J-CoMM, we further prove, for one inner sweep per outer iteration and under a set of standard compactness,
smoothness, and KL assumptions, convergence of the iterates to a critical point.

\item \textbf{Implementation and benchmarking.}
We describe practical dense/sparse contraction routines and an experimental protocol.
Our experiments show that, while per-iteration progress is comparable, \emph{joint majorization} can substantially
reduce wall-clock time by coupling cached reference tensors with contraction-only updates (notably for CP across the
tested $\beta$ values).
\end{itemize}

\paragraph{Paper organization.}
Section~\ref{sec:related_work} reviews related work and positions our contribution.
Section~\ref{sec:preliminaries} introduces notation, the $\beta$-divergence objective, the CP and Tucker models, and
recalls the majorization-minimization (MM) principle used throughout the paper.
Section~\ref{sec:block_mm} derives classical \emph{block} MM multiplicative updates written in a
\emph{contraction-only} form, i.e., without explicit unfoldings.
The proof details for these block majorizers and updates are provided in
Appendix~\ref{app:proof_block_mm}.
Section~\ref{sec:jmm_multilinear} presents our main contribution: a \emph{joint} majorizer built at a reference point
and decreased by inexpensive inner block updates, again using only tensor contractions.
It also contains a KL-based convergence analysis for J-CoMM with one inner sweep per outer iteration, under a set of
standard regularity assumptions.
Additional derivations and the blockwise-separable minimization arguments for the joint majorizer are given in
Appendix~\ref{app:proof_joint_mm} (and Section~\ref{app:jmm_scalar} for the scalar/blockwise minimization details).
Section~\ref{sec:algorithms} describes the resulting algorithms and discusses practical implementations of all
required contractions; explicit einsum recipes are collected in Appendix~\ref{app:einsum}.
For Tucker, fully indexed expressions of the quantities used in both block-MM and joint-MM are gathered in
Appendix~\ref{app:tucker_explicit}.
Section~\ref{sec:exp_protocol} reports numerical experiments on synthetic and real datasets.
Finally, Section~\ref{sec:conclusion} concludes and outlines future directions.

\section{Background and Related Work}\label{sec:related_work}

\paragraph{$\beta$-divergences and majorization-minimization}
The $\beta$-divergence family covers several standard data fitting losses, including squared Euclidean loss
($\beta=2$), the Kullback-Leibler (KL) divergence ($\beta=1$), and the Itakura-Saito (IS) divergence ($\beta=0$).
In nonnegative matrix factorization (NMF), these losses are often minimized with \emph{multiplicative updates} (MU).
The MU philosophy goes back to the early NMF literature and is popular because it is simple, fast, and preserves
nonnegativity by construction \cite{leeseung1999,leeseung2000}.
A key step was the MM derivation of MU for the $\beta$-divergence by F\'evotte and Idier \cite{fevotteidier2011},
which explains MU through explicit tight majorizers and provides monotonic decrease of the objective.
For KL-NMF, several algorithmic variants and practical improvements have been studied in detail
\cite{hiengillis2021kl,gillis2020nmf}.
More recently, BMMe adds a lightweight extrapolation mechanism on top of MM updates, with convergence guarantees
established in the matrix case (notably for $\beta\in[1,2]$) and strong empirical speedups \cite{bmme}.
Another recent line of work revisits MM through \emph{second-order} majorants: SOM/mSOM builds quadratic surrogate
functions based on Hessian bounds and discusses how monotonicity can be restored when global smoothness assumptions
fail near zero \cite{phamcohenchonavel2025som}.
Finally, beyond block-wise MM, joint-MM strategies have been developed for matrix $\beta$-NMF to reduce the cost of
rebuilding surrogates at every block update \cite{marmin2023joint}.

\paragraph{Tensor decompositions under divergence losses}
Nonnegative CP and Tucker decompositions are standard multilinear models, and many practical solvers follow the same
pattern as in the matrix case: they rely on block updates, and each block update reduces to a sequence of tensor
kernels (often implemented via unfoldings and MTTKRP-like operations).
Divergence losses have also been used in this multilinear setting.
Early work proposed nonnegative tensor factorization updates based on $\alpha$- and $\beta$-divergences
\cite{cichocki2007alphabeta}, and probabilistic/Bregman viewpoints were developed for tensor factorization models
\cite{yilmazcemgil2010,yilmaz2011gctf}.
In practice, unfolding-based computations can become expensive at scale because they may increase memory traffic and
materialize large intermediate arrays.
This motivates implementations that stay closer to the multilinear structure and use direct tensor contractions.

\paragraph{Einsum-based multiplicative updates beyond CP/Tucker}
A recent preprint by Hood and Schein \cite{hoodschein2026} proposes \emph{nonnegative einsum factorization}:
a user specifies a multilinear nonnegative model as an einsum string, and the method fits it with multiplicative
updates under a broad family of losses.
Their work highlights two points that strongly align with our motivation: tensor models can be implemented without
explicit unfoldings by relying on contraction primitives, and MM provides a clean route to monotonic descent and
convergence guarantees.

Our focus is more specialized. We study the nonnegative CP and Tucker models under the $\beta$-divergence and derive
updates in a form that is maximally explicit for these two canonical decompositions.
In particular, we provide contraction-only formulas tailored to CP/Tucker, and we introduce a \emph{joint}
majorization strategy that reuses reference-powered tensors across several inexpensive inner updates.

\paragraph{Joint MM for $\beta$-NMF and extension to multilinear models}
Marmin, de Morais Goulart, and F\'evotte \cite{marmin2023joint} introduced a joint majorization-minimization strategy
for matrix $\beta$-NMF.
Instead of rebuilding a surrogate for one block at a time, their approach constructs a single auxiliary function at a
reference iterate and decreases it via a small number of cheap inner updates, while keeping expensive reference
quantities fixed.
In this work, we adapt this joint-MM mechanism to nonnegative CP and Tucker decompositions.
The key technical requirement is to express all surrogate numerators and denominators as tensor contractions, so that
the resulting algorithms can be implemented efficiently with einsum primitives, without explicit unfoldings.

\section{Preliminaries}\label{sec:preliminaries}
We recall here the main definitions and tools required for the rest of the paper.
After setting notation, we introduce the entry-wise $\beta$-divergence that we use to measure the mismatch between a
nonnegative tensor and its reconstruction.
We then define the nonnegative CP and Tucker models, and recall the MM framework, which will be used to derive
both block-wise surrogates (leading to multiplicative updates) and our joint surrogate strategy.
We conclude with the \emph{Einstein summation} (einsum) notation, which will serve as our basic language to
implement all updates via \emph{tensor contractions} without explicit unfoldings.

\subsection{Notation}
Let $\X \in \R_+^{I_1\times\cdots\times I_N}$ be a nonnegative tensor.
We write $i=(i_1,\dots,i_N)$ for a multi-index and $\X_i$ for an entry.
For tensors (or matrices) of the same size, $\od$ and $\oslash$ denote elementwise product and elementwise division,
respectively.
For a tensor $\Y$, $\Y^\alpha$ denotes elementwise power.

\paragraph{Mode-$n$ product}
For a tensor $\Y\in\R^{I_1\times\cdots\times I_N}$ and a matrix $A\in\R^{J\times I_n}$, the mode-$n$ product
$\Y\times_n A\in\R^{I_1\times\cdots\times I_{n-1}\times J\times I_{n+1}\times\cdots\times I_N}$ is defined entrywise by
\[
(\Y\times_n A)_{i_1,\dots,i_{n-1},\,j,\,i_{n+1},\dots,i_N}
=
\sum_{i_n=1}^{I_n} \Y_{i_1,\dots,i_N}\,A_{j\,i_n}.
\]

\paragraph{Model reconstruction}
Given model parameters $\Theta$, we denote the reconstruction by $\Xhat(\Theta)\in\R_+^{I_1\times\cdots\times I_N}$.
When $\Theta$ is clear from context, we write simply $\Xhat$ and $\Xhat_i$ for its entries.

\subsection{\texorpdfstring{$\beta$}{beta}-divergence}
The objective function considered in this work is based on the $\beta$-divergence family, which we use to quantify the mismatch between $\X$ and its reconstruction $\Xhat$.
The loss is defined entrywise: we sum the scalar divergences $d_\beta(\X_i\mid \Xhat_i)$ over all indices, with $\beta$ controlling the discrepancy.
We now give the formal definition of the corresponding scalar $\beta$-divergence and of the resulting tensor objective $D_\beta(\X,\Xhat)$.

\begin{definition}[$\beta$-divergence]
For $x\ge 0$ and $y>0$, define $d_\beta(x\mid y)$ by
\[
d_\beta(x\mid y)=
\begin{cases}
\frac{1}{\beta(\beta-1)}\Big(x^\beta + (\beta-1)y^\beta - \beta x y^{\beta-1}\Big), & \beta\neq 0,1,\\[1mm]
x\log\frac{x}{y}-x+y, & \beta=1,\\[1mm]
\frac{x}{y}-\log\frac{x}{y}-1, & \beta=0.
\end{cases}
\]
For tensors, the objective is
\[
D_\beta(\X,\Xhat)=\sum_{i} d_\beta(\X_i\mid \Xhat_i).
\]
\end{definition}

\begin{remark}[The case $\beta=0$]\label{rem:beta0}
For the Itakura-Saito divergence ($\beta=0$), the quantity $d_0(x\mid y)$ is finite only when $x>0$ and $y>0$.
Accordingly, objective-value convergence statements at $\beta=0$ require either strictly positive data
$\X_i>0$ for all $i$, or the standard practical convention of evaluating the loss on a floored version of the data,
e.g., replacing $\X$ by $\max(\X,\varepsilon_X)$ for a small $\varepsilon_X>0$.
\end{remark}

\begin{assumption}[Positivity safeguard]
\label{ass:eps}
We enforce a small lower bound $\eps>0$ on all \emph{parameters} of the model, i.e.,
all factor matrices and the core tensor satisfy entrywise $\ge \eps$.
Consequently, the reconstructed tensor $\Xhat(\Theta)$ is strictly positive entrywise, so
all quantities of the form $\Xhat^{\beta-1}$ and $\Xhat^{\beta-2}$ are well-defined.
In the numerical implementation, we may additionally apply a small clipping
$\Xhat \leftarrow \max(\Xhat,\eps)$ \emph{only} when evaluating these powers, for numerical stability.
\end{assumption}

\subsection{Models}
We consider two standard constrained tensor decomposition models throughout the paper: the nonnegative canonical polyadic (CP) decomposition and the nonnegative Tucker decomposition. For completeness, we recall their definitions below.

\paragraph{Nonnegative CP (rank $R$)}
Let $\A^{(n)}\in\R_+^{I_n\times R}$ for $n=1,\dots,N$.
\[
\Xhat_{i_1,\dots,i_N}=\sum_{r=1}^R \prod_{n=1}^N \A^{(n)}_{i_n r}.
\]

\paragraph{Nonnegative Tucker}
Let $\G\in\R_+^{J_1\times\cdots\times J_N}$ and $\A^{(n)}\in\R_+^{I_n\times J_n}$.
\[
\Xhat = \G \times_1 \A^{(1)} \times_2 \cdots \times_N \A^{(N)}.
\]

\subsection{Majorization-Minimization Principle}

The updates proposed in this paper are derived within the majorization-minimization (MM) framework.
MM proceeds by constructing, at a current point $\widetilde{\theta}$, a surrogate function
$G(\theta\mid\widetilde{\theta})$ that upper bounds the objective and is tight at $\widetilde{\theta}$.
One then updates $\theta$ by decreasing (or minimizing) this surrogate, which guarantees monotonic decrease
of the original objective.
In Sections~\ref{sec:block_mm} and~\ref{sec:jmm_multilinear}, we use this principle in two ways:
(i) classical block surrogates that yield standard multiplicative updates, and
(ii) a joint surrogate built at a reference point and decreased by a few cheap inner updates.

For completeness, we next state the formal definition of a tight majorizing surrogate, followed by a key
descent proposition showing that MM updates generate a non-increasing sequence of objective function values.

\begin{definition}[Majorization-minimization surrogate]
Given an objective $F(\theta)$, a function $G(\theta\mid \tilde\theta)$ is a tight majorizing surrogate at $\tilde\theta$
if:
\begin{enumerate}
\item $G(\theta\mid \tilde\theta)\ge F(\theta)$ for all $\theta$,
\item $G(\tilde\theta\mid \tilde\theta)=F(\tilde\theta)$.
\end{enumerate}
\end{definition}

\begin{proposition}[Monotonic descent]
\label{prop:mm_descent}
If $\theta^{+}\in\arg\min_\theta G(\theta\mid \tilde\theta)$, then
$F(\theta^{+})\le F(\tilde\theta)$.
\end{proposition}

We will use surrogates that are separable in the entries of one block, which yields closed-form
multiplicative updates.
Later, we introduce a joint surrogate, which is built for all blocks at once but is minimized
by a few simple sub-iterations.

\subsection{Einstein summation (einsum) and contraction-only computations}
\label{subsec:prelim_einsum}
A central goal of this work is to express all numerators and denominators of our multiplicative updates as
\emph{tensor contractions}, so that they can be implemented without explicit matricizations (mode unfoldings).
To make this idea concrete, we briefly recall the Einstein summation notation, commonly exposed in numerical libraries
through the function \texttt{einsum}.

\paragraph{Basic rule}
An einsum expression specifies a product of tensors with explicit indices.
Indices that appear in the inputs but not in the output are summed out (contracted).
Equivalently, \texttt{einsum} provides a compact, index-based way to write ``multiply then sum over shared axes''.

\paragraph{Example 1: matrix multiplication}
Let $A\in\R^{I\times K}$ and $B\in\R^{K\times J}$.
The product $C=AB$ is
\[
C_{ij}=\sum_{k=1}^K A_{ik}B_{kj},
\]
which corresponds to the einsum string
\[
\texttt{'ik,kj->ij'}.
\]

\paragraph{Example 2: a CP contraction (third-order illustration)}
Let $\mathcal{T}\in\R^{I\times J\times K}$ and let
$B^{(2)}\in\R^{J\times R}$ and $B^{(3)}\in\R^{K\times R}$.
We define a contraction that produces a matrix $M\in\R^{I\times R}$ by
\begin{equation}
\label{eq:cp_contr_example}
M_{ir}
=\sum_{j=1}^J\sum_{k=1}^K \mathcal{T}_{ijk}\,B^{(2)}_{jr}\,B^{(3)}_{kr},
\qquad i=1,\dots,I,\ \ r=1,\dots,R.
\end{equation}
Three points are worth emphasizing.

\smallskip
\noindent\textbf{(i) Why the output is a matrix.}
In \eqref{eq:cp_contr_example}, the indices $j$ and $k$ are summed out, while $i$ and $r$ are \emph{free} indices.
Therefore the result is indexed by $(i,r)$ and has size $I\times R$.
Importantly, the index $r$ is \emph{not} summed: it labels the CP components and is carried to the output.

\smallskip
\noindent\textbf{(ii) Column-wise interpretation (one contraction per component).}
Let $b^{(2)}_r:=B^{(2)}_{:r}\in\R^{J}$ and $b^{(3)}_r:=B^{(3)}_{:r}\in\R^{K}$ denote the $r$th columns.
Then the $r$th column of $M$ is the vector $m^{(r)}\in\R^{I}$ defined by
\[
m^{(r)}_i=\sum_{j,k}\mathcal{T}_{ijk}\,b^{(2)}_{r}(j)\,b^{(3)}_{r}(k),
\qquad i=1,\dots,I,
\]
that is, $M=[m^{(1)}\ \cdots\ m^{(R)}]$.
This makes clear that we perform the same contraction for each component $r$, and stack the results.

\smallskip
\noindent\textbf{(iii) Relation to $n$-mode products.}
For a fixed $r$, the vector $m^{(r)}$ can be seen as successive mode products with \emph{vectors}:
\[
m^{(r)} \;=\; \mathcal{T}\times_2 (b^{(2)}_r)^\top \times_3 (b^{(3)}_r)^\top \;\in\; \R^{I}.
\]
The einsum form simply performs this computation for all $r$ simultaneously by keeping the component index $r$ in the
output.

\smallskip
\noindent\textbf{Einsum notation.}
Equation \eqref{eq:cp_contr_example} corresponds to
\[
\texttt{'ijk,jr,kr->ir'}.
\]
As a general rule, indices that appear in the inputs but not in the output are summed out (here: $j,k$), while indices
that appear in the output remain free (here: $i,r$).

\begin{remark}[When does a contraction produce an $I\times R\times R$ tensor?]
If we used two distinct component indices, for instance
\[
M_{irs}=\sum_{j,k}\mathcal{T}_{ijk}\,B^{(2)}_{jr}\,B^{(3)}_{ks},
\]
then the output would be indexed by $(i,r,s)$ and would have size $I\times R\times R$; the corresponding einsum is
\[
\texttt{'ijk,jr,ks->irs'}.
\]
In CP contractions we intentionally use the \emph{same} component index $r$ across modes, which yields an
$I\times R$ matrix (e.g., \texttt{'ijk,jr,kr->ir'}).
\end{remark}

The matrix contraction above is the basic computational primitive behind all CP updates in this paper:
it is used to form the MU numerators and denominators directly from the data tensor and the current factors,
without any explicit unfolding.
In particular, for a third-order tensor, the operator introduced later in Section~\ref{sec:algorithms} satisfies
\[
\mathrm{CPContr}^{(1)}(\mathcal{T};B^{(2)},B^{(3)}) \;=\; M,
\]
with entries given by \eqref{eq:cp_contr_example}; analogous expressions hold for modes $2$ and $3$ by permuting
indices.

We will rely on such contractions systematically: in CP, each block update reduces to computing one matrix
$\mathrm{CPContr}^{(n)}(\mathcal{P};\{A^{(m)}\}_{m\neq n})$ and one matrix
$\mathrm{CPContr}^{(n)}(\mathcal{Q};\{A^{(m)}\}_{m\neq n})$, both implemented as einsum calls.

Einsum provides a convenient and efficient abstraction for the contraction-only computations required in this work,
for two main reasons.
First, it allows us to implement CP/Tucker updates \emph{without} forming explicit unfoldings, Khatri-Rao products, or
large intermediate matrices; instead, we contract only along the indices that must be summed out.
Second, when a contraction involves several tensors, the order in which pairwise contractions are executed can have a
large impact on runtime and memory traffic; modern backends (e.g., by selecting optimized contraction paths) can
therefore yield substantial practical gains.
In the sequel, we systematically derive update formulas as contraction expressions, and later translate them into
explicit einsum recipes (Appendix~\ref{app:einsum}) for reproducible implementations.

\section{Separable Block Majorizers for \texorpdfstring{$\beta$}{beta}-Divergences}\label{sec:block_mm}

In this section we recall the standard block majorization-minimization (MM) construction for $\beta$-divergence
objectives and derive the resulting multiplicative updates.
We derive "contraction-only" update formulas, i.e., tensor expressions that compute the MU numerators and
denominators via direct contractions rather than explicit unfoldings.
These updates are algebraically identical to the classical MU rules.
We first introduce the two tensors $\mathcal P$ and $\mathcal Q$ that appear systematically in the gradients and
majorizers, and then present the CP and Tucker block updates.

\subsection{Weights and the multiplicative exponent}
The $\beta$-divergence gradients can be expressed using two entrywise ``weights'' built from the current reconstruction.
These tensors appear systematically in the numerators and denominators of the MU rules below.
Given a current reconstruction $\Xhat$, define
\[
\mathcal{P} := \X \od \Xhat^{\beta-2},
\qquad
\mathcal{Q} := \Xhat^{\beta-1},
\]
with elementwise powers, under the positivity safeguard.

\begin{remark}[Exponent $\gamma(\beta)$]
\label{rem:gamma}
For $\beta<2$, the MM scalar subproblems obtained from the surrogate have closed-form minimizers.
Solving the first-order optimality condition yields an update of the form
\[
U \leftarrow \widetilde U \odot \left(\frac{\mathrm{Num}}{\mathrm{Den}}\right)^{\gamma(\beta)},
\qquad
\gamma(\beta)=
\begin{cases}
\frac{1}{2-\beta}, & 0\le \beta<1,\\[1mm]
1, & 1\le \beta <2,
\end{cases}
\]
which matches the classical $\beta$-divergence MM derivations (see, e.g.,~\cite{fevotteidier2011}),
with the limit cases $\beta=0$ and $\beta=1$ understood by continuity.
In what follows we keep the notation $\gamma(\beta)$ and focus on $\beta\in[0,2)$.
\end{remark}

\begin{remark}[$\varepsilon$-constrained updates]
Throughout the paper, all block subproblems are understood on the
$\varepsilon$-constrained feasible set prescribed by Assumption~\ref{ass:eps}.
Hence, the exact block minimizer is obtained entrywise by the usual multiplicative
candidate followed by the lower bound $\varepsilon$, i.e.,
\[
U^{+}=\max\!\left(U_{\mathrm{cand}},\varepsilon\right)
\]
(entrywise).
When the unconstrained multiplicative candidate already satisfies
$U_{\mathrm{cand}}\ge \varepsilon$, the clipping is inactive and one recovers the
standard multiplicative update.
\end{remark}

With these definitions in place, we now derive separable block surrogates for each model.
The resulting updates take the generic form
\(
\text{block} \leftarrow \text{block} \od (\mathrm{Num}/\mathrm{Den})^{\gamma(\beta)},
\)
where $\mathrm{Num}$ and $\mathrm{Den}$ are obtained by contracting $\mathcal P$ and $\mathcal Q$ with model-specific
partial reconstructions.

\subsection{CP: block update without unfolding}
We first consider the CP model and update one factor matrix at a time while keeping the others fixed.
Because the reconstruction is linear in the active factor, the MM surrogate becomes separable across its entries,
leading to a closed-form multiplicative update.

Fix all factors except $\A^{(n)}$.
For each component $r$, define
\[
s^{(n)}_{r}(i_{-n}) := \prod_{m\neq n} \A^{(m)}_{i_m r}.
\]
Then $\Xhat_i=\sum_{r} \A^{(n)}_{i_n r}\,s^{(n)}_{r}(i_{-n})$ is linear in $\A^{(n)}$.
The next two quantities correspond to the standard MU numerator/denominator, but are written as contractions over the
index set $i_{-n}$ rather than via matricization.

Define numerator and denominator matrices of size $I_n\times R$:
\[
\mathrm{Num}^{(n)}_{i_n r}=\sum_{i_{-n}} \mathcal{P}_{i}\,s^{(n)}_{r}(i_{-n}),
\qquad
\mathrm{Den}^{(n)}_{i_n r}=\sum_{i_{-n}} \mathcal{Q}_{i}\,s^{(n)}_{r}(i_{-n}).
\]

\begin{theorem}[CP block multiplicative update]
\label{thm:block_cp}
Under Assumption~\ref{ass:eps}, the CP block subproblem in $\A^{(n)}$ admits a separable MM surrogate.
Its exact minimizer on the $\varepsilon$-constrained feasible set is obtained entrywise as
\[
\A^{(n)}_{+}
=
\max\!\left(
\A^{(n)} \odot
\left(\frac{\mathrm{Num}^{(n)}}{\mathrm{Den}^{(n)}}\right)^{\gamma(\beta)},
\ \eps
\right),
\]
where the maximum is taken entrywise.
In particular, when the unconstrained multiplicative candidate is already $\ge \eps$ entrywise,
the clipping is inactive and one recovers the standard multiplicative update
\[
\A^{(n)} \leftarrow \A^{(n)} \odot \left(\frac{\mathrm{Num}^{(n)}}{\mathrm{Den}^{(n)}}\right)^{\gamma(\beta)}.
\]
Moreover, each block update yields
\[
D_\beta(\X,\Xhat(\A^{(n)}_{+})) \le D_\beta(\X,\Xhat(\A^{(n)})).
\]
\end{theorem}
\begin{proof}
The proof follows the classical MM construction for $\beta$-divergence losses.
One first majorizes each entrywise term of the objective by Jensen's inequality and,
for $\beta<1$, combines this with the standard convex-concave split and a tangent upper bound.
Since the CP model is linear in the active block $\A^{(n)}$ when all other factors are fixed,
the resulting block surrogate is separable across the entries of $\A^{(n)}$.
Minimizing these scalar surrogate terms yields the stated multiplicative update, with the
entrywise lower bound $\eps$ enforced by clipping.
The detailed indexed derivation is given in Appendix~\ref{app:proof_block_mm}.
\end{proof}

\subsection{Tucker: block updates without unfolding}
For Tucker, the variables split into the core tensor and the factor matrices, and we derive a separable surrogate for
each block in turn.
We start with the core update (a multimode contraction with transposed factors) and then present the update for a
single factor matrix.

\paragraph{Core update}
Let $\Xhat=\G\times_1\A^{(1)}\cdots\times_N\A^{(N)}$ and define $\mathcal{P},\mathcal{Q}$ as above.
Define
\[
\mathcal{P}_{\mathrm{core}} := \mathcal{P}\times_1 (\A^{(1)})^\top \cdots \times_N (\A^{(N)})^\top,
\qquad
\mathcal{Q}_{\mathrm{core}} := \mathcal{Q}\times_1 (\A^{(1)})^\top \cdots \times_N (\A^{(N)})^\top.
\]
Then
\[
\G \leftarrow \G \od \left(\frac{\mathcal{P}_{\mathrm{core}}}{\mathcal{Q}_{\mathrm{core}}}\right)^{\gamma(\beta)}.
\]

\paragraph{Factor update}
The factor updates follow the same MU template: once all other blocks are fixed, the reconstruction becomes linear in
$\A^{(n)}$, and the corresponding numerator/denominator are obtained by contracting $\mathcal P$ and $\mathcal Q$ with
the partial tensor $\mathcal B^{(n)}$.

Fix all blocks except $\A^{(n)}$ and define the partial tensor
\[
\mathcal{B}^{(n)} := \G \times_1 \A^{(1)}\cdots\times_{n-1}\A^{(n-1)}
\times_{n+1}\A^{(n+1)}\cdots\times_N\A^{(N)}.
\]
Then
\[
\Xhat_{i} = \sum_{j_n} \A^{(n)}_{i_n j_n}\, \mathcal{B}^{(n)}_{j_n,i_{-n}}.
\]
Define
\[
\mathrm{Num}^{(n)}_{i_n j_n}=\sum_{i_{-n}} \mathcal{P}_{i}\, \mathcal{B}^{(n)}_{j_n,i_{-n}},
\qquad
\mathrm{Den}^{(n)}_{i_n j_n}=\sum_{i_{-n}} \mathcal{Q}_{i}\, \mathcal{B}^{(n)}_{j_n,i_{-n}},
\]
and update
\[
\A^{(n)} \leftarrow \A^{(n)} \od \left(\frac{\mathrm{Num}^{(n)}}{\mathrm{Den}^{(n)}}\right)^{\gamma(\beta)}.
\]

We summarize the resulting monotonicity property in the following theorem; the proof follows the same MM argument as
for CP and is deferred to the appendix.

\begin{theorem}[Tucker block multiplicative updates]
\label{thm:block_tucker}
Under Assumption~\ref{ass:eps}, the Tucker core and factor subproblems admit separable MM surrogates.
Their exact minimizers on the $\varepsilon$-constrained feasible set are obtained entrywise by applying
the usual multiplicative candidates followed by the lower bound $\eps$.
Equivalently,
\[
\G_{+}
=
\max\!\left(
\G \odot \left(\frac{\mathcal{P}_{\mathrm{core}}}{\mathcal{Q}_{\mathrm{core}}}\right)^{\gamma(\beta)},
\ \eps
\right),
\]
and, for each mode $n$,
\[
\A^{(n)}_{+}
=
\max\!\left(
\A^{(n)} \odot \left(\frac{\mathrm{Num}^{(n)}}{\mathrm{Den}^{(n)}}\right)^{\gamma(\beta)},
\ \eps
\right),
\]
where the maxima are taken entrywise.
When the unconstrained multiplicative candidates already satisfy the lower bound,
the clipping is inactive and one recovers the standard multiplicative updates.
Moreover, each such block update decreases the objective $D_\beta(\X,\Xhat)$.
\end{theorem}

\begin{proof}
The argument is the same as in the CP case.
When all other blocks are fixed, the Tucker model is linear in the active block, whether this
block is the core tensor or one factor matrix.
Applying the standard $\beta$-divergence MM construction therefore yields a separable surrogate
for the active block, and the corresponding scalar minimizers give the stated multiplicative
updates, followed by the entrywise lower bound $\eps$ when needed.
The full indexed derivation and the associated contraction formulas are provided in
Appendix~\ref{app:proof_block_mm}.
\end{proof}

\section{Joint Majorizers for Multilinear Models}
\label{sec:jmm_multilinear}

Block updates are simple and known to be efficient in practice, but they can be slow when each block update requires
costly recomputation of intermediate tensors.
Joint majorization introduces a \emph{single} auxiliary function for all variables at once.
This auxiliary function is built at a reference point, and it is then decreased by a few cheap inner
sub-iterations while keeping the expensive reference tensors fixed.
This is the main mechanism we use to reduce runtime.

A key conceptual point is that the surrogate is constructed \emph{jointly} for all blocks (it upper-bounds the full
objective in all variables), but its separability is \emph{blockwise}:
when all blocks except one are fixed, the surrogate becomes entrywise separable in the active block.
This conditional separability yields closed-form inner multiplicative updates for both CP and Tucker within a single
template.

\subsection{Joint surrogate and inner sub-iterations}
Let $\Theta$ denote all variables of the model.
For CP, $\Theta=\{\A^{(n)}\}_{n=1}^N$.
For Tucker, $\Theta=\{\G,\A^{(1)},\dots,\A^{(N)}\}$.

Fix a reference point $\widetilde{\Theta}$, and write $\widetilde{\Xhat}=\widehat{\X}(\widetilde{\Theta})$.
A joint surrogate is a function $G(\Theta\mid\widetilde{\Theta})$ such that
\[
G(\Theta\mid\widetilde{\Theta}) \ge D_\beta(\X,\widehat{\X}(\Theta)),
\qquad
G(\widetilde{\Theta}\mid\widetilde{\Theta}) = D_\beta(\X,\widehat{\X}(\widetilde{\Theta})).
\]
Since $G(\cdot\mid\widetilde{\Theta})$ majorizes the objective and is tight at $\widetilde{\Theta}$, any decrease of
$G(\cdot\mid\widetilde{\Theta})$ yields a decrease of a valid upper bound on the true objective.
In our scheme, this guarantees objective decrease \emph{across outer iterations} (see
Theorem~\ref{thm:jmm_monotone}); the objective is not necessarily guaranteed to decrease after each inner step.

Each inner block update minimizes $G(\cdot\mid\widetilde\Theta)$ with respect to the active block
(with the other blocks fixed), \emph{in exact arithmetic via the closed-form minimizer} of the corresponding
scalar subproblems, so $G(\cdot\mid\widetilde\Theta)$ decreases monotonically throughout the inner loop;
see Appendix~\ref{app:proof_joint_mm} (Section~\ref{subsec:Inner_Decrease_and_monot}).

\subsection{How the joint surrogate is constructed}
\label{subsec:jmm_construct}
This subsection explains what $G(\Theta\mid\widetilde{\Theta})$ looks like and why it is useful.

\paragraph{Step 1. Write the model as a sum of nonnegative contributions}
For every tensor index $i=(i_1,\dots,i_N)$, we write the model entry as
\[
\widehat{\X}(\Theta)_i = \sum_{\rho\in\mathcal{R}} z_{i,\rho}(\Theta),
\qquad
z_{i,\rho}(\Theta)\ge 0.
\]
For CP we use $\mathcal{R}=\{1,\dots,R\}$ and $z_{i,r}(\Theta)=\prod_{n=1}^N A^{(n)}_{i_n r}$.
For Tucker we use $\mathcal{R}=\{1,\dots,J_1\}\times\cdots\times\{1,\dots,J_N\}$ and
$z_{i,j}(\Theta)=\mathcal{G}_{j_1\dots j_N}\prod_{n=1}^N A^{(n)}_{i_n j_n}$.

\paragraph{Step 2. Define reference weights}
At the reference iterate $\widetilde{\Theta}$, define
\[
\widetilde{\lambda}_{i,\rho}
:=\frac{z_{i,\rho}(\widetilde{\Theta})}{\widetilde{\Xhat}_i},
\qquad
\sum_{\rho\in\mathcal{R}}\widetilde{\lambda}_{i,\rho}=1,
\qquad
\widetilde{\Xhat}_i>0.
\]
These weights depend only on the reference, so they can be reused during the inner loop.
Under Assumption~\ref{ass:eps}, all contributions $z_{i,\rho}(\widetilde{\Theta})$ are strictly positive,
hence $\widetilde{\lambda}_{i,\rho}>0$ and the ratios $z_{i,\rho}(\Theta)/\widetilde{\lambda}_{i,\rho}$ are well-defined.

\paragraph{Step 3. Use Jensen and tangency}
The main difficulty is the term $d_\beta(\X_i\mid \widehat{\X}(\Theta)_i)$, where the second argument is a sum.
When $\beta\in[1,2)$, the map $y\mapsto d_\beta(x\mid y)$ is convex on $y>0$, so we can apply Jensen:
\[
d_\beta\!\Big(\X_i \,\Big|\, \sum_{\rho} z_{i,\rho}(\Theta)\Big)
\le
\sum_{\rho} \widetilde{\lambda}_{i,\rho}\,
 d_\beta\!\Big(\X_i \,\Big|\, \frac{z_{i,\rho}(\Theta)}{\widetilde{\lambda}_{i,\rho}}\Big).
\]
When $\beta\in[0,1)$, the standard approach is to split $d_\beta$ into a convex part and a concave part with
respect to $y$, apply Jensen to the convex part, and upper bound the concave part by its tangent at
$\widetilde{\Xhat}_i$.
The result is again an upper bound that is a sum over $\rho$.

\paragraph{Entrywise joint surrogate and global surrogate}
In both cases, we obtain an entrywise upper bound
\[
d_\beta(\X_i\mid \widehat{\X}(\Theta)_i) \le G_i(\Theta\mid \widetilde{\Theta}),
\qquad
G_i(\widetilde{\Theta}\mid \widetilde{\Theta}) = d_\beta(\X_i\mid \widetilde{\Xhat}_i),
\]
and we define
\[
G(\Theta\mid \widetilde{\Theta}) := \sum_i G_i(\Theta\mid \widetilde{\Theta}).
\]
This $G$ is tight at the reference and is built using weights $\widetilde{\lambda}_{i,\rho}$.

A complete derivation of the joint surrogate (including the case $\beta\in[0,1)$, and the blockwise separability
steps leading to the multiplicative inner updates) is given in Appendix~\ref{app:proof_joint_mm}.

\subsection{Closed-form inner updates that decrease the joint surrogate}
\label{subsec:jmm_inner_updates}
The joint surrogate is built at $\widetilde{\Theta}$, but it is decreased by simple block updates.
The key point is \emph{conditional} separability: with all other blocks fixed,
$G(\Theta\mid \widetilde{\Theta})$ becomes separable in the entries of the selected block, which yields multiplicative updates.

For completeness, Appendix~\ref{app:proof_joint_mm} (Section~\ref{app:jmm_scalar}) makes this statement fully explicit:
for CP (and similarly Tucker), once all other blocks are fixed, the joint surrogate
$G(\cdot\mid\widetilde\Theta)$ decomposes into a sum of independent one-dimensional convex functions over the
entries of the active block. Their unique minimizers yield exactly the inner multiplicative updates stated below;
see Lemma~\ref{lem:jmm_cp_scalar} and Corollary~\ref{cor:jmm_tucker_scalar}.

\paragraph{Reference-powered tensors}
We define
\[
\widetilde{\mathcal{P}} := \X \od \widetilde{\Xhat}^{\beta-2},
\qquad
\widetilde{\mathcal{Q}} := \widetilde{\Xhat}^{\beta-1},
\]
with elementwise powers and the safeguard $\widetilde{\Xhat}_i\ge \eps$.
These tensors are fixed during the inner loop.

\paragraph{Two transforms}
For a nonnegative variable $Z$ with reference $\widetilde{Z}$, define
\[
\chi_{1,\beta}(Z,\widetilde Z)= \widetilde Z^{\,2-\beta}\odot Z^{\,\beta-1},
\qquad
\chi_{2,\beta}(Z,\widetilde Z)=
\begin{cases}
Z, & \beta<1,\\
Z^{\,\beta}\odot \widetilde Z^{-(\beta-1)}, & 1\le \beta<2.
\end{cases}
\]
These are applied entrywise to factor matrices and to the Tucker core. Note that
$\chi_{1,\beta}(Z,Z)=Z$ and $\chi_{2,\beta}(Z,Z)=Z$.

\paragraph{CP inner update}
Fix $\widetilde{\Theta}=\{\widetilde A^{(n)}\}$ and let $\Theta=\{A^{(n)}\}$ be the current inner iterate.
For each mode $n$, define
\[
\mathrm{Num}^{(n)}_{\mathrm{J}}(i_n,r)
=
\sum_{i_{-n}}
\widetilde{\mathcal{P}}_{i}\,
\prod_{m\neq n}\chi_{1,\beta}\!\left(A^{(m)}_{i_m r},\widetilde A^{(m)}_{i_m r}\right),
\]
\[
\mathrm{Den}^{(n)}_{\mathrm{J}}(i_n,r)
=
\sum_{i_{-n}}
\widetilde{\mathcal{Q}}_{i}\,
\prod_{m\neq n}\chi_{2,\beta}\!\left(A^{(m)}_{i_m r},\widetilde A^{(m)}_{i_m r}\right).
\]
Then the inner update for the $n$th CP factor is
\[
A^{(n)} \leftarrow \widetilde A^{(n)} \od \left(\frac{\mathrm{Num}^{(n)}_{\mathrm{J}}}{\mathrm{Den}^{(n)}_{\mathrm{J}}}\right)^{\gamma(\beta)}.
\]
All terms are computed by tensor contractions, without explicit unfoldings.
This update is the unique minimizer of $G(\cdot\mid\widetilde\Theta)$ with respect to the block
$A^{(n)}$ when all other factors are fixed; see Appendix~\ref{app:proof_joint_mm}, Section~\ref{app:jmm_scalar} (Lemma~\ref{lem:jmm_cp_scalar}).

\paragraph{Tucker inner updates}
Fix $\widetilde{\Theta}=\{\widetilde{\G},\widetilde A^{(1)},\dots,\widetilde A^{(N)}\}$ and let
$\Theta=\{\G,A^{(1)},\dots,A^{(N)}\}$ be the current inner iterate.
The core update is
\[
\G \leftarrow \widetilde{\G} \od \left(\frac{\widetilde{\mathcal{P}}_{\mathrm{core,J}}}{\widetilde{\mathcal{Q}}_{\mathrm{core,J}}}\right)^{\gamma(\beta)},
\]
where $\widetilde{\mathcal{P}}_{\mathrm{core,J}}$ and $\widetilde{\mathcal{Q}}_{\mathrm{core,J}}$ are obtained by
$n$-mode contractions using the transformed factors $\chi_{1,\beta}(A^{(n)},\widetilde A^{(n)})$ and
$\chi_{2,\beta}(A^{(n)},\widetilde A^{(n)})$, respectively.
Similarly, for each mode $n$ the factor update has the form
\[
A^{(n)} \leftarrow \widetilde A^{(n)} \od \left(\frac{\mathrm{Num}^{(n)}_{\mathrm{J}}}{\mathrm{Den}^{(n)}_{\mathrm{J}}}\right)^{\gamma(\beta)},
\]
with numerator and denominator computed by contraction-only operations involving
$\widetilde{\mathcal{P}}$, $\widetilde{\mathcal{Q}}$, the current inner blocks, and the reference blocks.
As in the CP case, these updates (for both the core and the factor matrices) are the unique minimizers of
$G(\cdot\mid\widetilde\Theta)$ with respect to the considered block, with all other blocks fixed; see
Appendix~\ref{app:proof_joint_mm}, Section~\ref{app:jmm_scalar} (Corollary~\ref{cor:jmm_tucker_scalar}).

\begin{proposition}[Inner updates decrease the joint surrogate]
\label{prop:inner_decrease}
Under Assumption~\ref{ass:eps}, the inner multiplicative updates described above are the unique minimizers of
$G(\cdot\mid \widetilde{\Theta})$ with respect to the updated block (holding all other blocks fixed).
In particular, each inner update satisfies
\[
G(\Theta^{+}\mid\widetilde{\Theta}) \le G(\Theta\mid\widetilde{\Theta}).
\]
\end{proposition}

\begin{proof}
The result follows from the combination of the global frozen-reference majorizer and the
scalar structure of the inner updates.

Appendix~\ref{app:proof_joint_mm} establishes that $G(\cdot\mid\widetilde\Theta)$ is a valid
majorizer of the objective and is tight at the reference point.
Section~\ref{app:jmm_scalar} of that appendix then shows that, once the reference quantities
are fixed, each block subproblem decomposes into independent scalar surrogate minimizations,
and that the J-CoMM update is the exact minimizer of these scalar terms.
Consequently, each inner block update does not increase the frozen-reference surrogate
$G(\cdot\mid\widetilde\Theta)$.
Combining these blockwise decreases over the inner sweep yields the stated descent property.
\end{proof}

\begin{remark}
The role of the reference point is to make $\widetilde{\mathcal{P}}$ and $\widetilde{\mathcal{Q}}$ fixed during the
inner loop.
This allows us to reuse expensive intermediate tensors across several block updates.
\end{remark}

\subsection{Monotonic decrease of the objective for joint MM}
\begin{theorem}[Monotonic decrease across outer iterations for joint MM]
\label{thm:jmm_monotone}
Let $\widetilde{\Theta}$ be the reference at an outer iteration and initialize the inner loop at
$\Theta^{(0)}=\widetilde{\Theta}$.
Assume the inner loop produces $\Theta^{(L)}$ such that
\[
G(\Theta^{(L)}\mid\widetilde{\Theta}) \le G(\widetilde{\Theta}\mid\widetilde{\Theta}).
\]
Then the objective decreases across the outer iteration:
\[
D_\beta(\X,\widehat{\X}(\Theta^{(L)})) \le D_\beta(\X,\widehat{\X}(\widetilde{\Theta})).
\]
In particular, if the outer update sets $\Theta^{(t+1)}:=\Theta^{(L)}$ with $\widetilde{\Theta}=\Theta^{(t)}$,
then the outer objective sequence is nonincreasing.
\end{theorem}

\begin{proof}
By majorization, for any $\Theta$,
\[
D_\beta(\X,\widehat{\X}(\Theta)) \le G(\Theta\mid\widetilde{\Theta}).
\]
Therefore,
\[
D_\beta(\X,\widehat{\X}(\Theta^{(L)}))
\le G(\Theta^{(L)}\mid\widetilde{\Theta})
\le G(\widetilde{\Theta}\mid\widetilde{\Theta})
= D_\beta(\X,\widehat{\X}(\widetilde{\Theta})).
\]
\end{proof}

\begin{remark}
The inner loop is guaranteed to decrease the fixed surrogate $G(\cdot\mid\widetilde{\Theta})$.
However, the original objective $D_\beta(\X,\widehat{\X}(\Theta))$ is only guaranteed to decrease
between outer iterates (from $\widetilde{\Theta}$ to $\Theta^{(L)}$), not necessarily after each inner update.
\end{remark}

\subsection{Convergence of objective values}
\begin{remark}[Lower boundedness]
For $x\ge 0$ and $y>0$, the $\beta$-divergence satisfies $d_\beta(x\mid y)\ge 0$, with equality iff $x=y$.
Hence $D_\beta(\X,\Xhat)\ge 0$, so the objective is bounded below on the feasible set.
\end{remark}

\begin{theorem}[Convergence of the objective values]
\label{thm:obj_values_conv}
Consider either block majorization-minimization (Section~\ref{sec:block_mm}) or joint
majorization-minimization (Section~\ref{sec:jmm_multilinear}).
Assume Assumption~\ref{ass:eps} and that the initial objective value is finite
(in particular, for $\beta=0$, see the remark~\ref{rem:beta0}).
Then:
\begin{itemize}
\item for block MM, the objective value decreases after each block update;
\item for joint MM, the objective value decreases after each outer iteration.
\end{itemize}
In both cases, the corresponding sequence of objective values converges to a finite limit.
\end{theorem}

\begin{proof}
For block MM, monotonic decrease follows from Proposition~\ref{prop:mm_descent} applied to each block update.
For joint MM, monotonic decrease across outer iterations follows from Theorem~\ref{thm:jmm_monotone}.
In both cases, the objective values form a monotone sequence bounded below, hence they converge.
\end{proof}

\subsection{Convergence of iterates and connection with BSUM}
Convergence of objective values does not by itself imply convergence of iterates.
For \emph{block MM}, the method is naturally related to the BSUM framework~\cite{bsum}:
each block update minimizes a block surrogate that is tight at the current iterate and upper bounds the objective
with respect to the active block.

To invoke a standard BSUM result rigorously, one must verify the usual assumptions for the block surrogate
(exactness, upper-bound property, continuity, and first-order consistency at the current iterate), and ensure that
the iterates remain in a compact set.
For CP and Tucker models, compactness is typically enforced by standard normalization steps that remove the scaling
indeterminacies while preserving the reconstruction.

\begin{remark}[Compactness and scaling indeterminacy]
For CP and Tucker models, scaling transformations can leave $\Xhat$ unchanged.
Accordingly, to obtain compact level sets one typically augments the algorithm with a normalization convention
(e.g., column normalizations with compensating rescaling in another block).
The corresponding stationary-point statements should then be understood for the normalized
$\varepsilon$-constrained formulation.
\end{remark}

\begin{proposition}[Stationary accumulation points for block MM via BSUM]
\label{prop:blockmm_stationary}
Assume that:
(i) the iterates remain in a compact subset of the feasible set
(e.g., after a standard normalization removing scaling indeterminacies),
(ii) the objective is continuous and regular on that set, and
(iii) each block update is the unique minimizer of a tight upper bound for that block
(on the $\varepsilon$-constrained feasible set).
Then standard BSUM theory implies that every accumulation point of the block-MM sequence is a stationary point of the
normalized $\varepsilon$-constrained problem.
Equivalently, the distance from the iterates to the set of stationary points tends to zero.
\end{proposition}

\begin{remark}[We do not claim iterate convergence for joint MM]
The joint-MM scheme uses a surrogate built at a \emph{fixed} reference $\widetilde\Theta$ and performs several inner
block minimizations before refreshing the surrogate.
At the beginning of the inner sweep, the frozen surrogate $G(\cdot\mid \widetilde\Theta)$ is tight at the current iterate
$\widetilde\Theta=\Theta^k$.
However, after the first block update, the current inner iterate generally differs from $\widetilde\Theta$, while the
surrogate remains frozen at $\widetilde\Theta$.
Hence it is no longer tight at the current inner iterate, so J-CoMM does not directly fit the standard BSUM template.
A sharper iterate-convergence analysis for J-CoMM, based on sufficient decrease, a relative-error estimate, and the
KL property, is developed in Section~\ref{subsec:jcomm_kl}.
\end{remark}

\subsection{Iterate convergence of J-CoMM for one inner sweep: a KL-based analysis}
\label{subsec:jcomm_kl}

The BSUM framework does not directly apply to J-CoMM because the joint surrogate
$G(\cdot\mid\widetilde\Theta)$ is constructed at a reference point and then kept fixed
during the inner sweep; in particular, it is generally not tight at the current inner iterate.
To analyze J-CoMM beyond monotonicity of objective values, we therefore follow a different route,
combining the MM viewpoint with a Kurdyka-\L ojasiewicz (KL) descent framework
\cite{beckpan2018,attouchboltesvaiter2013,boltesabachteboulle2014}.

In this subsection, we restrict attention to the practically most relevant case $L=1$,
that is, one inner sweep per outer iteration.
This setting is also the easiest one theoretically: each outer step is then a finite cyclic sweep of
exact block minimizations of a fixed surrogate built at the previous iterate.

Let
\[
F(\Theta):=D_\beta(\X,\widehat{\X}(\Theta)),
\qquad
\Psi(\Theta):=F(\Theta)+\iota_{\mathcal C}(\Theta),
\]
where $\mathcal C$ is a closed feasible set, and $\iota_{\mathcal C}$ denotes its indicator function.

Throughout this subsection, $\partial$ denotes the limiting (Mordukhovich) subdifferential.
Since $\mathcal C$ is closed and $F$ is continuous on a neighborhood of $\mathcal C$, the function
$\Psi=F+\iota_{\mathcal C}$ is proper and lower semicontinuous.
We call $\Theta^\star$ a critical point of $\Psi$ if
\[
0\in \partial \Psi(\Theta^\star).
\]

The proof proceeds in four steps.
First, we show that the scalar block subproblems defining the J-CoMM updates are uniformly strongly convex
on the considered feasible set.
Second, this yields a sufficient decrease estimate for one outer J-CoMM step.
Third, we prove a relative-error bound showing that the first-order residual at the new iterate is controlled
by the step length.
Finally, combining these two estimates with the KL property yields convergence of the whole sequence to a critical point.

For the sake of compactness, detailed proofs for the new KL-based convergence analysis of J-CoMM are deferred to Appendix~\ref{app:jcomm_kl_proofs}.

We first recall the KL property and then state the assumptions under which the J-CoMM iterate-convergence
analysis is carried out.

\begin{definition}[Kurdyka-\L ojasiewicz property]
A proper lower semicontinuous function $\Phi$ is said to satisfy the Kurdyka-\L ojasiewicz (KL) property at
$\bar x\in \operatorname{dom}\partial \Phi$ if there exist $\eta>0$, a neighborhood $U$ of $\bar x$, and a continuous
concave function $\varphi:[0,\eta)\to\R_+$ such that
\begin{enumerate}
\item $\varphi(0)=0$,
\item $\varphi$ is $C^1$ on $(0,\eta)$,
\item $\varphi'(s)>0$ for all $s\in(0,\eta)$,
\end{enumerate}
and for all $x\in U$ satisfying
\[
\Phi(\bar x)<\Phi(x)<\Phi(\bar x)+\eta,
\]
one has
\[
\varphi'\bigl(\Phi(x)-\Phi(\bar x)\bigr)\,\operatorname{dist}(0,\partial \Phi(x))\ge 1.
\]
We say that $\Phi$ is a KL function if it satisfies the KL property at every point of $\operatorname{dom}\partial\Phi$.
\end{definition}

\begin{assumption}[Standing assumptions for the J-CoMM convergence analysis]
\label{ass:jcomm_conv}
In this subsection, we assume that:
\begin{enumerate}
\item J-CoMM is run with one inner sweep per outer iteration ($L=1$);
\item the feasible set has a block-product structure
      \[
      \mathcal C=\mathcal C_1\times\cdots\times \mathcal C_B,
      \]
      where each $\mathcal C_b$ is closed and convex, and the iterates remain in the compact set $\mathcal C$;
\item there exist constants $0<\eps\le M<\infty$ such that every entry of every admissible block satisfies
      \[
      \eps \le \Theta_\alpha \le M
      \qquad \text{for all } \Theta\in\mathcal C;
      \]
\item the data tensor is strictly positive entrywise on the considered domain:
      \[
      \X_i \ge \delta_X >0 \qquad \text{for all } i;
      \]
\item the objective $F$ is continuously differentiable on an open neighborhood of $\mathcal C$,
      and its gradient is Lipschitz continuous on $\mathcal C$ with constant $L_F$;
\item for every fixed reference point $\widetilde\Theta\in\mathcal C$, the joint surrogate
      $\Theta\mapsto G(\Theta\mid\widetilde\Theta)$ is continuously differentiable on an open neighborhood of
      $\mathcal C$, and its gradient with respect to the first argument is uniformly Lipschitz on $\mathcal C$:
      there exists $L_G>0$ such that
      \[
      \|\nabla_1 G(\Theta\mid\widetilde\Theta)-\nabla_1 G(\Theta'\mid\widetilde\Theta)\|
      \le L_G\|\Theta-\Theta'\|
      \qquad\text{for all }\Theta,\Theta',\widetilde\Theta\in\mathcal C;
      \]
\item the joint surrogate is first-order consistent at the reference point:
      \[
      \nabla_1 G(\widetilde\Theta\mid \widetilde\Theta)=\nabla F(\widetilde\Theta)
      \qquad\text{for all }\widetilde\Theta\in\mathcal C;
      \]
\item the constrained objective $\Psi=F+\iota_{\mathcal C}$ satisfies the KL property on $\mathcal C$.
\end{enumerate}
\end{assumption}

\begin{remark}[On the assumptions of the J-CoMM convergence theorem]
The assumptions above play different roles.
Assumptions on the block structure and compactness of $\mathcal C$ are imposed to remove scaling indeterminacies and
to obtain uniform bounds.
The strict positivity of the data is a technical condition used to ensure that the scalar J-CoMM surrogate coefficients
remain uniformly bounded away from zero.
By contrast, the smoothness assumptions on $F$ and on the joint surrogate are natural for the $\beta$-divergence
family on the positive compact domain induced by the positivity safeguard.
Finally, the KL assumption is standard in nonconvex optimization and is automatically satisfied in many important
settings, notably when the constrained objective is semialgebraic; this covers, in particular, the rational values
of $\beta$ considered in our experiments.
\end{remark}

\begin{lemma}[Uniform curvature of the scalar J-CoMM block surrogates]
\label{lem:jcomm_uniform_curvature}
Assume Assumption~\ref{ass:jcomm_conv}.
Consider one outer J-CoMM step with reference iterate $\widetilde\Theta\in\mathcal C$,
and fix one active scalar variable of a factor matrix or of the Tucker core.
Let
\[
u=\frac{Z}{\widetilde Z}
\]
denote the corresponding ratio variable, where $\widetilde Z$ is the reference value and $Z$ the current value.

Then the associated scalar surrogate subproblem has the form
\[
g(u)=
\begin{cases}
\displaystyle
\frac{\mathrm{Den}}{\beta}\,u^\beta
-\frac{\mathrm{Num}}{\beta-1}\,u^{\beta-1}
+\mathrm{const},
& 1<\beta<2,\\[3mm]
\displaystyle
\mathrm{Den}\,u-\mathrm{Num}\log u+\mathrm{const},
& \beta=1,\\[3mm]
\displaystyle
\mathrm{Den}\,u+\frac{\mathrm{Num}}{1-\beta}\,u^{\beta-1}
+\mathrm{const},
& 0\le\beta<1,
\end{cases}
\]
where $\mathrm{Num}>0$ and $\mathrm{Den}>0$ denote the corresponding J-CoMM contraction coefficients.

Moreover, there exist constants
\[
0<\underline u\le \overline u<\infty,
\qquad
0<\underline N\le \mathrm{Num}\le \overline N<\infty,
\qquad
0<\underline D\le \mathrm{Den}\le \overline D<\infty,
\]
depending only on $\mathcal C$, $\eps$, $M$, $\delta_X$, and $\beta$, such that
\[
u\in[\underline u,\overline u]
\]
for every admissible scalar block update, and there exists a constant $\mu>0$ such that
\[
g''(u)\ge \mu
\qquad \text{for all admissible }u.
\]
In particular, each scalar surrogate is uniformly strongly convex on its admissible interval,
admits a unique minimizer $u^\star$, and satisfies
\[
g(u)-g(u^\star)\ge \frac{\mu}{2}\,|u-u^\star|^2
\qquad\text{for all }u\in[\underline u,\overline u].
\]
\end{lemma}

\begin{proof}
The proof consists in writing each scalar J-CoMM subproblem in the ratio variable
$u=Z/\widetilde Z$ and computing its second derivative explicitly in the three regimes
$0\le \beta<1$, $\beta=1$, and $1<\beta<2$.
Under Assumption~\ref{ass:jcomm_conv}, the admissible variables and the coefficients
$\mathrm{Num}$ and $\mathrm{Den}$ remain in compact positive intervals, which yields a
uniform lower bound on $g''(u)$.
The details are given in Appendix~\ref{app:jcomm_kl_proofs},
Section~\ref{app:proof_jcomm_uniform_curvature}.
\end{proof}

\begin{lemma}[Sufficient decrease for one outer J-CoMM step]
\label{lem:jcomm_sufficient_decrease}
Assume Assumption~\ref{ass:jcomm_conv}.
Let $\{\Theta^k\}$ be the sequence generated by J-CoMM with one inner sweep per outer iteration ($L=1$).
Let $B$ denote the number of blocks ($B=N$ for CP and $B=N+1$ for Tucker, counting the core as one block).
For one outer iteration $k$, define the intermediate iterates
\[
\Theta^{k,0}:=\Theta^k,\qquad
\Theta^{k,b}\quad (b=1,\dots,B),
\]
where $\Theta^{k,b}$ denotes the state after updating the first $b$ blocks of the fixed surrogate
$G(\cdot\mid \Theta^k)$, so that
\[
\Theta^{k+1}=\Theta^{k,B}.
\]

Then there exists a constant $c_{\mathrm{dec}}>0$, independent of $k$, such that
\[
F(\Theta^k)-F(\Theta^{k+1})
\;\ge\;
c_{\mathrm{dec}}
\sum_{b=1}^B \|\Theta^{k,b}-\Theta^{k,b-1}\|^2
\;=\;
c_{\mathrm{dec}}\|\Theta^{k+1}-\Theta^k\|^2,
\]
where $\|\cdot\|$ denotes the product Euclidean/Frobenius norm over all blocks.

Since all iterates remain feasible, the same inequality holds with $\Psi$ in place of $F$:
\[
\Psi(\Theta^k)-\Psi(\Theta^{k+1})
\;\ge\;
c_{\mathrm{dec}}\|\Theta^{k+1}-\Theta^k\|^2.
\]
\end{lemma}

\begin{proof}
For each block update within one outer sweep, the uniform curvature established in
Lemma~\ref{lem:jcomm_uniform_curvature} yields a quadratic lower bound on the decrease of the
fixed surrogate $G(\cdot\mid\Theta^k)$.
Summing these blockwise decreases over the sweep and using the majorization relation between
$G(\cdot\mid\Theta^k)$ and $F$ gives the stated sufficient decrease estimate.
The full argument is given in Appendix~\ref{app:jcomm_kl_proofs},
Section~\ref{app:proof_jcomm_sufficient_decrease}.
\end{proof}

\begin{lemma}[Relative-error bound for one outer J-CoMM step]
\label{lem:jcomm_relative_error}
Assume Assumption~\ref{ass:jcomm_conv}.
Let $\{\Theta^k\}$ be the sequence generated by J-CoMM with one inner sweep per outer iteration ($L=1$).
Then there exists a constant $c_{\mathrm{err}}>0$, independent of $k$, such that
\[
\operatorname{dist}\bigl(0,\partial \Psi(\Theta^{k+1})\bigr)
\le
c_{\mathrm{err}}\,\|\Theta^{k+1}-\Theta^k\|.
\]
\end{lemma}

\begin{proof}
The proof combines the first-order optimality conditions for the exact block minimizations of
the frozen-reference surrogate with the first-order consistency
$\nabla_1 G(\widetilde\Theta\mid\widetilde\Theta)=\nabla F(\widetilde\Theta)$ and the
Lipschitz continuity assumptions on $\nabla F$ and $\nabla_1 G$.
This yields a bound of the subgradient residual at $\Theta^{k+1}$ in terms of the step norm
$\|\Theta^{k+1}-\Theta^k\|$.
See Appendix~\ref{app:jcomm_kl_proofs},
Section~\ref{app:proof_jcomm_relative_error}.
\end{proof}

\begin{proposition}[Asymptotic regularity and critical cluster points]
\label{prop:jcomm_cluster}
Assume Assumption~\ref{ass:jcomm_conv}.
Let $\{\Theta^k\}$ be the sequence generated by J-CoMM with one inner sweep per outer iteration ($L=1$).
Then:

\begin{enumerate}
\item the sequence $\{\Psi(\Theta^k)\}$ is nonincreasing and converges to a finite limit, denoted by $\Psi_\infty$;
\item the increments are square-summable:
\[
\sum_{k=0}^\infty \|\Theta^{k+1}-\Theta^k\|^2 < \infty;
\]
in particular,
\[
\|\Theta^{k+1}-\Theta^k\|\to 0;
\]
\item every cluster point of $\{\Theta^k\}$ is a critical point of $\Psi$.
\end{enumerate}
\end{proposition}

\begin{proof}
The conclusion follows by combining the sufficient decrease estimate of
Lemma~\ref{lem:jcomm_sufficient_decrease} with the relative-error bound of
Lemma~\ref{lem:jcomm_relative_error}.
The former implies monotonicity of $\Psi(\Theta^k)$ and square summability of the increments,
while the latter shows that any cluster point satisfies the criticality condition.
The details are given in Appendix~\ref{app:jcomm_kl_proofs},
Section~\ref{app:proof_jcomm_cluster}.
\end{proof}

\begin{theorem}[Convergence of J-CoMM for one inner sweep]
\label{thm:jcomm_convergence}
Assume Assumption~\ref{ass:jcomm_conv}.
Let $\{\Theta^k\}$ be the sequence generated by J-CoMM with one inner sweep per outer iteration ($L=1$).
Then the sequence has finite length, namely
\[
\sum_{k=0}^\infty \|\Theta^{k+1}-\Theta^k\| < \infty,
\]
and therefore converges to a critical point $\Theta^\star$ of the constrained objective
\[
\Psi(\Theta)=F(\Theta)+\iota_{\mathcal C}(\Theta).
\]
\end{theorem}

\begin{proof}
By Assumption~\ref{ass:jcomm_conv}, the sequence $\{\Theta^k\}$ is contained in the compact set $\mathcal C$.
By Proposition~\ref{prop:jcomm_cluster}, the objective values $\Psi(\Theta^k)$ decrease to a finite limit $\Psi_\infty$,
the increments satisfy $\|\Theta^{k+1}-\Theta^k\|\to 0$, and every cluster point of the sequence is a critical point of $\Psi$.

Moreover, since $F$ is continuous on $\mathcal C$ and all iterates are feasible, for any convergent subsequence
$\Theta^{k_j}\to \Theta^\star$ we have
\[
\Psi(\Theta^{k_j})=F(\Theta^{k_j}) \to F(\Theta^\star)=\Psi(\Theta^\star).
\]
Thus the standard assumptions of the KL convergence theorem for descent sequences are satisfied:
the sufficient decrease estimate of Lemma~\ref{lem:jcomm_sufficient_decrease},
the relative-error estimate of Lemma~\ref{lem:jcomm_relative_error},
and the continuity condition along cluster subsequences.
Since $\Psi$ satisfies the KL property on $\mathcal C$, the standard KL convergence theorem
\cite[Theorem~2.9 and the finite-length argument]{attouchboltesvaiter2013}
applies.

Consequently,
\[
\sum_{k=0}^\infty \|\Theta^{k+1}-\Theta^k\| < \infty.
\]
Hence $\{\Theta^k\}$ is a Cauchy sequence, and therefore converges in the ambient finite-dimensional space:
\[
\Theta^k \to \Theta^\star .
\]
Because $\mathcal C$ is closed and $\Theta^k\in\mathcal C$ for all $k$, we have $\Theta^\star\in\mathcal C$.
Since $\Theta^\star$ is the limit of the sequence, it is in particular a cluster point; therefore,
by Proposition~\ref{prop:jcomm_cluster},
\[
0\in \partial \Psi(\Theta^\star).
\]
\end{proof}

\section{Algorithms and Efficient Tensor Contractions}\label{sec:algorithms}
This section has two roles.
First, it describes the algorithms we will test.
Second, it explains how all required quantities can be computed without explicit unfoldings.

\subsection{Contraction operators}
For CP it is convenient to define a contraction operator that maps a tensor and factor matrices to
an $I_n\times R$ matrix.
Given a tensor $\mathcal{T}\in\R^{I_1\times\cdots\times I_N}$ and matrices
$B^{(m)}\in\R^{I_m\times R}$ for $m\neq n$, define
\[
\mathrm{CPContr}^{(n)}(\mathcal{T};\{B^{(m)}\}_{m\neq n}) \in \R^{I_n\times R}
\]
entrywise by
\[
\big[\mathrm{CPContr}^{(n)}(\mathcal{T};\{B^{(m)}\}_{m\neq n})\big]_{i_n r}
=
\sum_{i_{-n}} \mathcal{T}_{i}\prod_{m\neq n} B^{(m)}_{i_m r},
\]
where the summation $\sum_{i_{-n}}$ runs over all indices $(i_1,\dots,i_{n-1},i_{n+1},\dots,i_N)$.
This contraction can be implemented directly using einsum primitives.

For Tucker we use standard mode-$n$ products $\times_n$.
In practice, we implement them as contractions to avoid explicit unfoldings.

\subsection{Block MM baseline}
The block MM baseline (for the CP model) is summarized in Algorithm~\ref{alg:block_cp}, which we call \textit{B-CoMM}
(\textbf{B}lock \textbf{Co}ntraction-only \textbf{M}ajorization-\textbf{M}inimization).
(An analogous contraction-only block MM baseline for Tucker follows the same pattern and is omitted for brevity.)

\begin{algorithm}[ht!]
\caption{\textbf{B-CoMM}: Block MM for CP under $\beta$-divergence, contraction-only}
\label{alg:block_cp}
\begin{algorithmic}[1]
\Require $\X\ge 0$, rank $R$, initial factors $\{\A^{(n)}\ge \eps\}$, $\beta\in[0,2)$
\For{$k=0,1,2,\dots$}
  \For{$n=1$ to $N$}
    \State Compute current reconstruction $\Xhat$ from the CP model
    \State Form powered tensors with the positivity safeguard of Assumption~\ref{ass:eps}:
    $\mathcal{P}\gets \X \od \Xhat^{\beta-2}$,\quad $\mathcal{Q}\gets \Xhat^{\beta-1}$
    \State $\mathrm{Num}^{(n)} \gets \mathrm{CPContr}^{(n)}(\mathcal{P};\{\A^{(m)}\}_{m\neq n})$
    \State $\mathrm{Den}^{(n)} \gets \mathrm{CPContr}^{(n)}(\mathcal{Q};\{\A^{(m)}\}_{m\neq n})$
    \State $\A^{(n)} \gets \A^{(n)} \od \left(\mathrm{Num}^{(n)}\oslash \mathrm{Den}^{(n)}\right)^{\gamma(\beta)}$
    \State Enforce $\A^{(n)}\ge \eps$ entrywise
  \EndFor
\EndFor
\end{algorithmic}
\end{algorithm}

\subsection{How to implement the contractions with einsum}
\label{subsec:einsum_practice}
We now explain concretely what we mean by \emph{einsum-based contractions}.
The idea is to write the desired summation with explicit indices and to let an einsum backend
(e.g., \texttt{numpy.einsum} or \texttt{opt\_einsum}) perform the contraction without forming
any unfolding, Khatri-Rao product, or Kronecker product.

To keep notation readable, we illustrate the pattern on third-order tensors.
Complete recipes, including the joint-MM contractions, are listed in Appendix~\ref{app:einsum}.

\paragraph{CP contraction (third-order example)}
Let $\mathcal{T}\in\R^{I\times J\times K}$ and let $B^{(2)}\in\R^{J\times R}$,
$B^{(3)}\in\R^{K\times R}$.
The CP contraction $\mathrm{CPContr}^{(1)}(\mathcal{T};B^{(2)},B^{(3)})\in\R^{I\times R}$ is
\[
\big[\mathrm{CPContr}^{(1)}(\mathcal{T};B^{(2)},B^{(3)})\big]_{i r}
=
\sum_{j=1}^J\sum_{k=1}^K \mathcal{T}_{i j k}\,B^{(2)}_{j r}\,B^{(3)}_{k r},
\]
implemented in Python as
\texttt{einsum('ijk,jr,kr->ir', T, B2, B3, optimize=True)}.

\paragraph{Tucker reconstruction and core contraction (third-order example)}
Let $\G\in\R^{J_1\times J_2\times J_3}$ and factor matrices
$A^{(1)}\in\R^{I_1\times J_1}$, $A^{(2)}\in\R^{I_2\times J_2}$, $A^{(3)}\in\R^{I_3\times J_3}$.
The Tucker reconstruction
\[
\Xhat_{i_1 i_2 i_3}=\sum_{j_1,j_2,j_3}\G_{j_1 j_2 j_3}\,
A^{(1)}_{i_1 j_1}A^{(2)}_{i_2 j_2}A^{(3)}_{i_3 j_3}
\]
is \texttt{einsum('abc,ia,jb,kc->ijk', G, A1, A2, A3, optimize=True)}.
For $\mathcal{P}\in\R^{I_1\times I_2\times I_3}$, the core contraction
$\mathcal{P}_{\mathrm{core}}=\mathcal{P}\times_1(A^{(1)})^\top\times_2(A^{(2)})^\top\times_3(A^{(3)})^\top$
is \texttt{einsum('ijk,ia,jb,kc->abc', P, A1, A2, A3, optimize=True)}.

Appendix~\ref{app:tucker_explicit} provides explicit indexed formulas for all Tucker numerators and
denominators (block and joint), and Appendix~\ref{app:einsum} translates them into einsum calls.

\subsection{Joint MM algorithm}
The joint method follows an outer/inner structure.
At each outer iteration we build the reference reconstruction and the corresponding reference-powered tensors.
Then, for a small number of inner steps, we update each block while keeping the reference fixed, which enables
reuse of the reference tensors across several updates.

Our joint method is summarized in Algorithm~\ref{alg:joint_general}, which we call \textit{J-CoMM}
(\textbf{J}oint \textbf{Co}ntraction-only \textbf{M}ajorization-\textbf{M}inimization).

\begin{algorithm}[ht!]
\caption{\textbf{J-CoMM}: Joint MM for CP or Tucker under $\beta$-divergence, contraction-only}
\label{alg:joint_general}
\begin{algorithmic}[1]
\Require $\X\ge 0$, model parameters $\Theta^{(0)}\ge \eps$, $\beta\in[0,2)$, inner steps $L\ge 1$
\For{$t=0,1,2,\dots$}
  \State Set reference $\widetilde{\Theta}\gets \Theta^{(t)}$
  \State Compute reference reconstruction $\widetilde{\Xhat}\gets \widehat{\X}(\widetilde{\Theta})$
  \State Form reference-powered tensors with the positivity safeguard of Assumption~\ref{ass:eps}:
  $\widetilde{\mathcal{P}}\gets \X \od \widetilde{\Xhat}^{\beta-2}$,\quad
  $\widetilde{\mathcal{Q}}\gets \widetilde{\Xhat}^{\beta-1}$
  \State Initialize inner iterate $\Theta\gets \widetilde{\Theta}$
  \For{$\ell=1$ to $L$}
    \State For each block, build transformed factors from the current inner iterate $\Theta$
    using $\chi_{1,\beta}(\cdot,\widetilde{\cdot})$ and $\chi_{2,\beta}(\cdot,\widetilde{\cdot})$
    \State Update block
    $\Theta_b \gets \widetilde{\Theta}_b \od \left(\mathrm{Num}_{b,\mathrm{J}}\oslash \mathrm{Den}_{b,\mathrm{J}}\right)^{\gamma(\beta)}$
    using contractions with $\widetilde{\mathcal{P}}$ and $\widetilde{\mathcal{Q}}$
    \State Enforce positivity $\Theta\ge \eps$ entrywise
  \EndFor
  \State Set $\Theta^{(t+1)}\gets \Theta$
\EndFor
\end{algorithmic}
\end{algorithm}

In all experiments reported in Section~\ref{sec:exp_protocol}, we use a single inner step per outer iteration
($L=1$), so each reported J-CoMM iteration corresponds to one full sweep over the blocks. This choice is also the one covered by the iterate-convergence analysis of Section~\ref{subsec:jcomm_kl}.

\paragraph{For sparse tensors}
When $\X$ is sparse, contractions involving $\X$ (e.g., numerators based on $\mathcal{P}=\X\od \Xhat^{\beta-2}$)
can be accumulated efficiently by looping over nonzero entries.
In contrast, denominator terms involve $\mathcal{Q}=\Xhat^{\beta-1}$ and are typically dense because they depend on
the model values $\Xhat$.
In special cases (e.g., $\beta=1$ where $\mathcal{Q}\equiv \mathbf{1}$) they simplify substantially.

\subsection{Going beyond: Majorization-Minimization with Extrapolation}
\label{sec:mm_extrap}

A standard block majorization-minimization (block-MM) method updates one block at a time by
minimizing a block surrogate (majorizer) built at the current iterate.
The extrapolated block-MM framework (BMMe)~\cite{bmme} modifies \emph{only} the evaluation point of the surrogate:
before updating a block, one first forms an extrapolated (inertial) point from the two most recent iterates,
and then performs the same MM step but with the surrogate built (or evaluated) at this extrapolated point.
This can yield a noticeable acceleration at essentially negligible extra cost per iteration, since extrapolation is
just an elementwise operation on the parameters.

More precisely, for a block variable $x_i\in X_i$, BMMe forms
\[
\widehat x_i^{\,t} \;=\; x_i^{t} + \alpha_i^{t}\, P_i\!\big(x_i^{t}-x_i^{t-1}\big),
\]
and then performs the block-MM update
\[
x_i^{t+1}\in \arg\min_{x_i\in X_i} \; G_i^t\!\big(x_i \mid \widehat x_i^{\,t}\big),
\]
where $G_i^t(\cdot\mid \cdot)$ is a valid majorizer for the block objective (with other blocks fixed).
In the nonnegativity-constrained setting, a natural choice is
$P_i(\Delta)=[\Delta]_+$ (componentwise positive part), together with the usual positivity safeguard
$x\leftarrow\max(x,\varepsilon)$, consistent with Assumption~\ref{ass:eps}.

BMMe uses extrapolation sequences inspired by accelerated (mirror) descent, combined with a safeguard that
controls the extrapolation displacement.
A practical choice is a Nesterov-like sequence $\alpha_t^{\text{Nes}}$,
together with a cap depending on the displacement norm:
\[
\alpha_i^t \;=\; \min\!\left(\alpha_t^{\text{Nes}},\;\frac{c_t}{\|P_i(x_i^{t}-x_i^{t-1})\|+\delta}\right),
\]
where $\delta>0$ is a small constant and $(c_t)$ is a slowly decreasing or bounded sequence.
(Unless stated otherwise, we use the same scalar $\alpha_t$ for all blocks in our implementations.)

\smallskip
\noindent
Our contraction-only updates are closed-form multiplicative rules obtained from the same MM principle as classical
(unfolding-based) MU.
This suggests a natural BMMe-inspired extrapolation mechanism for our contraction-only updates:
one replaces the current block iterate by its extrapolated version when
(i) building the powered tensors used in the numerator/denominator contractions and
(ii) applying the multiplicative update (using the extrapolated block as the multiplicative pre-factor).

\smallskip
\noindent\textbf{(i) Extrapolated B-CoMM (CP).}
Maintain, for each factor, the previous outer iterate $A_{\mathrm{prev}}^{(n)}$.
Before updating mode $n$, form the extrapolated block
\[
\widehat A^{(n)} \;=\; \max\!\big(A^{(n)} + \alpha_t [A^{(n)}-A_{\mathrm{prev}}^{(n)}]_+,\;\varepsilon\big).
\]
Build the reconstruction $\widehat X$ using $\widehat A^{(n)}$ (and the other factors at their current values),
form $\mathcal P = X\odot \widehat X^{\beta-2}$ and $\mathcal Q=\widehat X^{\beta-1}$ using the positivity safeguard
of Assumption~\ref{ass:eps}, and apply the same contraction-only MU step, anchored at $\widehat A^{(n)}$:
\[
A^{(n)} \;\leftarrow\;
\max\!\left(
\widehat A^{(n)} \odot
\left(
\frac{\mathrm{CPContr}^{(n)}(\mathcal P;\{A^{(m)}\}_{m\neq n})}
     {\mathrm{CPContr}^{(n)}(\mathcal Q;\{A^{(m)}\}_{m\neq n})}
\right)^{\gamma(\beta)}
,\;\varepsilon\right).
\]
Finally set $A_{\mathrm{prev}}^{(n)}\leftarrow A^{(n)}$ at the end of the outer iteration.

\smallskip
\noindent\textbf{(ii) Extrapolated B-CoMM (Tucker).}
The same modification applies to each factor block $A^{(n)}$ and to the core $G$:
before updating a block, extrapolate it using $[\cdot]_+$ and $\alpha_t$, reconstruct $\widehat X$ with that
extrapolated block, and use $\widehat A^{(n)}$ (or $\widehat G$) as the multiplicative pre-factor in the
corresponding numerator/denominator update, based on the Tucker contractions in Appendix~\ref{app:tucker_explicit}.

\smallskip
\noindent\textbf{(iii) Heuristic outer-reference extrapolation for J-CoMM.}
A simple way to combine extrapolation with caching is to extrapolate \emph{the outer reference}:
\[
\widetilde\Theta \;\leftarrow\; \max\!\big(\Theta^{(t)} + \alpha_t[\Theta^{(t)}-\Theta^{(t-1)}]_+,\;\varepsilon\big),
\]
then build $\widetilde{\widehat X}$, $\widetilde{\mathcal P}$ and $\widetilde{\mathcal Q}$ from $\widetilde\Theta$,
and run the inner joint-MM steps exactly as in Algorithm~\ref{alg:joint_general}.
This preserves the main benefit of J-CoMM (reuse of $\widetilde{\mathcal P},\widetilde{\mathcal Q}$ across inner updates),
while injecting inertial acceleration at the outer level.

BMMe provides subsequence convergence guarantees under multi-convexity assumptions and suitable choices of
$\alpha_t$; in particular, the BMMe analysis in~\cite{bmme} covers matrix $\beta$-NMF for $\beta\in[1,2]$.
In our tensor setting, the extrapolation mechanisms above should therefore be viewed as
BMMe-inspired extensions of the contraction-only MM updates.
For B-CoMM, they preserve the same closed-form contraction structure and add only negligible
implementation cost.
For J-CoMM, outer-reference extrapolation is more heuristic: it preserves the caching
benefit of $\widetilde{\mathcal P}$ and $\widetilde{\mathcal Q}$, but it is not covered by the present
KL-based convergence result for noninertial J-CoMM.
Accordingly, we use these extrapolated variants as practical acceleration mechanisms in the
experiments, without claiming a convergence theory for them in the present paper.

\section{Numerical Tests}
\label{sec:exp_protocol}

\paragraph{Goal, loss, and reported metrics}
All methods minimize the entry-wise $\beta$-divergence between the input tensor $X$ and its reconstruction $\widehat X$.
In all plots, we report the \emph{mean} $\beta$-divergence per entry,
\[
\bar D_\beta(X,\widehat X) \;:=\; \frac{1}{|X|}\, D_\beta(X,\widehat X),
\qquad |X|=\prod_{n=1}^N I_n,
\]
for $\beta\in\{1/2,\,1,\,3/2\}$, and we plot $\bar D_\beta(X,\widehat X)$ versus (i) the iteration index and (ii) the
wall-clock CPU time.
Our implementation returns $D_\beta(X,\widehat X)$ as a sum over entries and is normalized by $|X|$ before plotting,
whereas NNEinFact reports a mean loss natively; we set its parameters to $\alpha=1$ and $\beta_{\mathrm{AB}}=\beta-1$
so that the reported loss matches the $\beta$-divergence used in our methods.
In the Uber benchmark, the loss is evaluated over all entries (mask of all ones), so all methods are compared on the
same objective and the same scale.

\paragraph{On the interpretation of the iteration axis}
For the unfolding-based MU baseline and for B-CoMM, one iteration corresponds to one full sweep over all blocks.
For J-CoMM, one outer iteration consists of one surrogate refresh followed by $L$ inner block sweeps; in all reported
experiments we set $L=1$.
Hence, all methods perform one full block sweep per reported iteration.
That said, the arithmetic work per sweep is not identical across methods, since J-CoMM reuses a fixed reference
surrogate and reference-powered tensors within the sweep, whereas block-MM recomputes the current powered quantities
block by block.
For this reason, wall-clock time remains the primary fairness metric.

\paragraph{Compared methods}
We benchmark the following implementations:
\begin{enumerate}
\item \textbf{Unfolding-based MU (baseline).}
We implement classical multiplicative updates (MU) through tensor unfoldings (matricizations) and associated
MTTKRP\footnote{MTTKRP stands for \emph{Matricized Tensor Times Khatri-Rao Product}. It is the standard kernel in CP
computations: given a mode-$n$ unfolding $X_{(n)}$ and the Khatri-Rao product of the other factor matrices, MTTKRP
computes the matrix product $X_{(n)}(\odot_{k\neq n} A^{(k)})$.}
computations, following the MU lineage of Lee and Seung~\cite{leeseung1999} and the general $\beta$-divergence MM
derivation of F\'evotte and Idier~\cite{fevotteidier2011}.
After each multiplicative update, we apply a standard truncation safeguard,
$U \leftarrow \max(U,\varepsilon)$ entrywise (with a small $\varepsilon>0$), to prevent zero-locking and improve numerical
robustness in finite precision arithmetic~\cite{gillis2020nmf}.
This truncation can also be interpreted as enforcing entrywise constraints on all model parameters
(e.g., factor matrices and, for Tucker, the core), $\Theta \ge \varepsilon$;
in this setting, modified MU schemes enjoy monotonicity and convergence-to-stationary-point guarantees
for the $\varepsilon$-constrained objective; see~\cite{gillis2020nmf}.
In practice, MU methods are widely regarded as among the most efficient approaches for $\beta$-divergence NMF-type
models; see, e.g.,~\cite{gillis2020nmf,hiengillis2021kl,bmme}.

\item \textbf{Block-MM (ours).} The contraction-only block MM baseline in Algorithm~\ref{alg:block_cp}.

\item \textbf{Joint-MM (ours).} The contraction-only joint MM scheme in Algorithm~\ref{alg:joint_general}.

\item \textbf{NNEinFact.} The general einsum-based framework of~\cite{hoodschein2026}.
\end{enumerate}
All methods are initialized with the same random seed within each run (so that data generation and initialization are
consistent across methods).

\subsection{Implementation details and fairness of runtime comparisons}
\label{subsec:impl_fairness}

\paragraph{Python/NumPy vs.\ PyTorch implementations}
Our implementations (B-CoMM, J-CoMM) and the unfolding-based MU baseline are written in Python/NumPy.
NNEinFact is implemented in PyTorch and executes the main contractions via optimized tensor kernels; on CPU,
these kernels can leverage intra-operation multithreading (controlled by \texttt{torch.set\_num\_threads}).
To provide a fair and informative comparison, we run NNEinFact with three CPU thread settings ($1$, $4$, and $8$) and
report all three runtime curves.
In addition, we report a single-thread reference configuration in which PyTorch is set to one thread and
NumPy/BLAS threading is restricted to one thread, so that all methods are compared under the same threading budget.

\paragraph{Reproducibility}
An online Colab demo (CP and Tucker, including B-CoMM and J-CoMM) is available at:
\url{https://colab.research.google.com/drive/1vyXcP76_XGRoNdCxQxRCPPV3MwvgY8vv?usp=sharing}.
The full Python codebase, including scripts to reproduce every figure reported in this paper, is available at:
\url{https://github.com/vleplat/CoMM.git}.
All experiments were run on a MacBook Pro (M4 Pro, 24~GB memory).
Reported runtimes are wall-clock CPU measurements intended for relative comparisons between methods; absolute timings
may differ on other CPU architectures and on GPU.

\subsection{Synthetic data}

\paragraph{Synthetic CP}
We generate a $4$th-order nonnegative tensor of size $80\times70\times60\times50$ from a rank-$R$ CP model with $R=10$. For each random seed, we sample nonnegative ground-truth factor matrices, form the corresponding CP tensor $X$,
and run every algorithm from a random nonnegative initialization.
We repeat this procedure over $5$ seeds and plot the mean performance curve, with a lightly shaded region
showing the empirical variability (standard deviation) across seeds.
We repeat the full benchmark for each $\beta\in\{1/2,\,1,\,3/2\}$; the corresponding results are reported in
Figure~\ref{fig:synth_cp}.

\begin{figure*}[ht!]
\centering

\noindent\textbf{Synthetic CP, $\beta=0.5$}\par\vspace{1mm}
\includegraphics[width=0.98\textwidth]{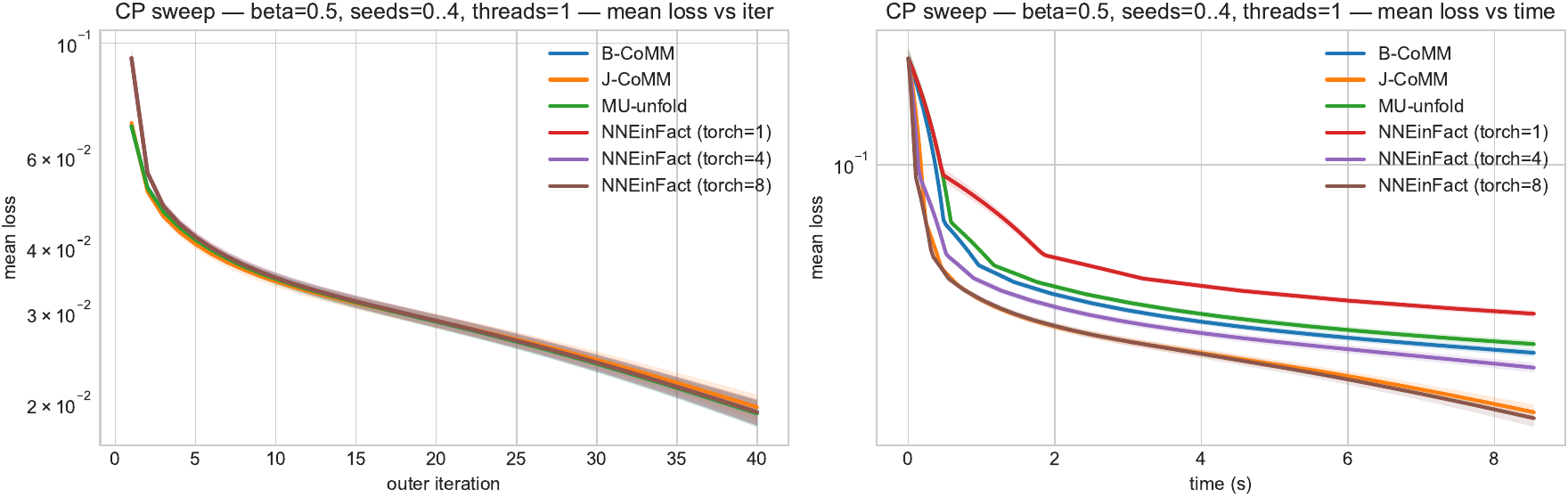}

\vspace{2mm}
\noindent\textbf{Synthetic CP, $\beta=1$}\par\vspace{1mm}
\includegraphics[width=0.98\textwidth]{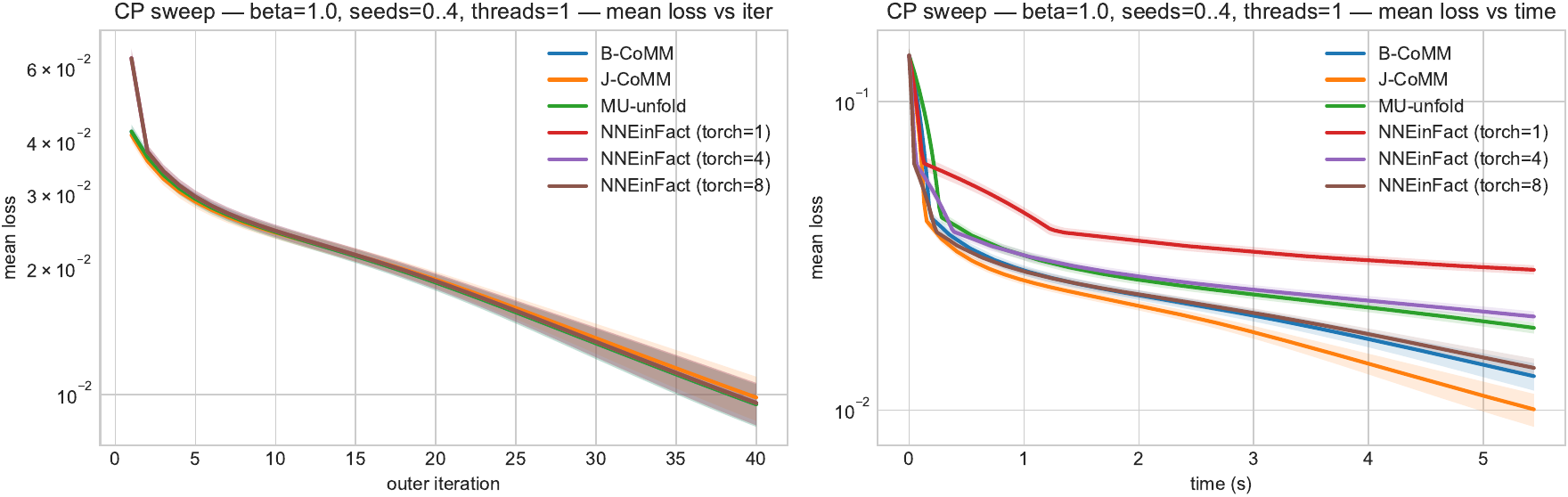}

\vspace{2mm}
\noindent\textbf{Synthetic CP, $\beta=1.5$}\par\vspace{1mm}
\includegraphics[width=0.98\textwidth]{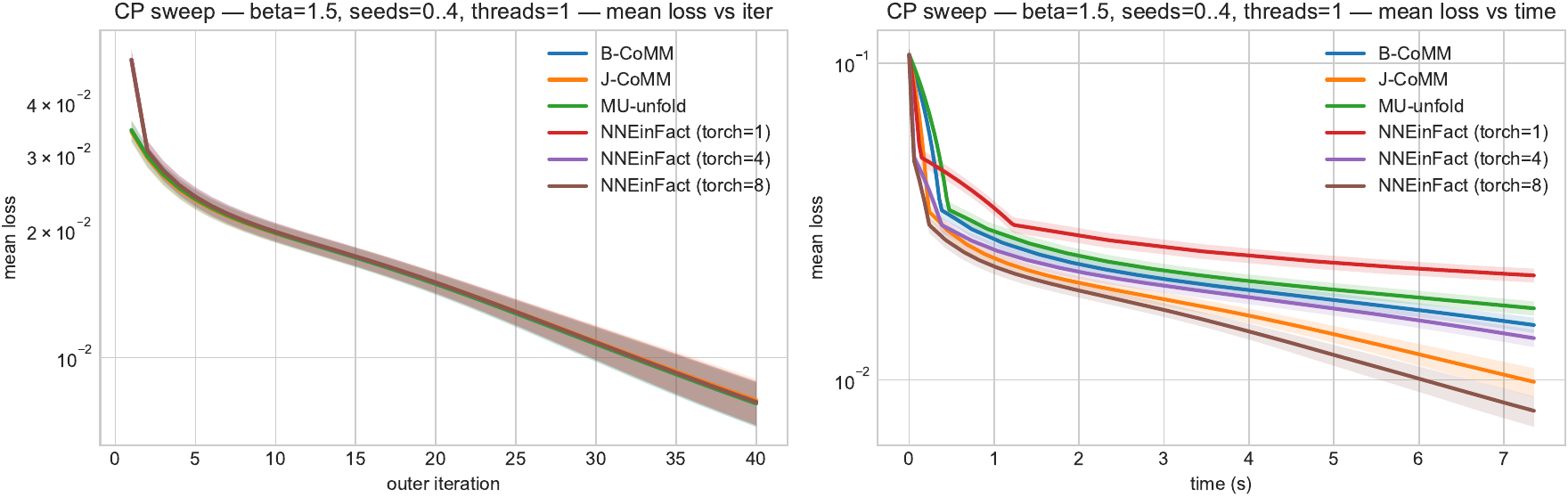}

\caption{Synthetic CP benchmark (order $4$, size $80{\times}70{\times}60{\times}50$, rank $R=10$).
Each row corresponds to one $\beta\in\{0.5,1,1.5\}$ and shows the mean normalized loss $\bar D_\beta(X,\widehat X)$
versus iteration (left) and wall-clock CPU time (right), averaged over $5$ random seeds (shaded band: variability).
For runtime fairness, NumPy/BLAS is restricted to one thread; NNEinFact is additionally reported for Torch CPU threads
$1/4/8$ (see legend).}
\label{fig:synth_cp}
\end{figure*}

\paragraph{Synthetic Tucker}
We repeat the same protocol for a Tucker model.
For each seed, we sample a nonnegative core tensor $G\in\R_+^{10\times 9\times 8\times 7}$ and nonnegative factor
matrices $A^{(1)}\in\R_+^{80\times 10}$, $A^{(2)}\in\R_+^{70\times 9}$, $A^{(3)}\in\R_+^{60\times 8}$, and
$A^{(4)}\in\R_+^{50\times 7}$, construct the input tensor as
$X = G \times_1 A^{(1)} \times_2 A^{(2)} \times_3 A^{(3)} \times_4 A^{(4)}$,
and run each method from a random nonnegative initialization.
We aggregate results over $5$ seeds by reporting the mean curve and a light-shaded band for variability, and we repeat
the experiment for $\beta\in\{1/2,\,1,\,3/2\}$; the corresponding results are reported in
Figure~\ref{fig:synth_tucker}.

\begin{figure*}[ht!]
\centering

\noindent\textbf{Synthetic Tucker, $\beta=0.5$}\par\vspace{1mm}
\includegraphics[width=0.98\textwidth]{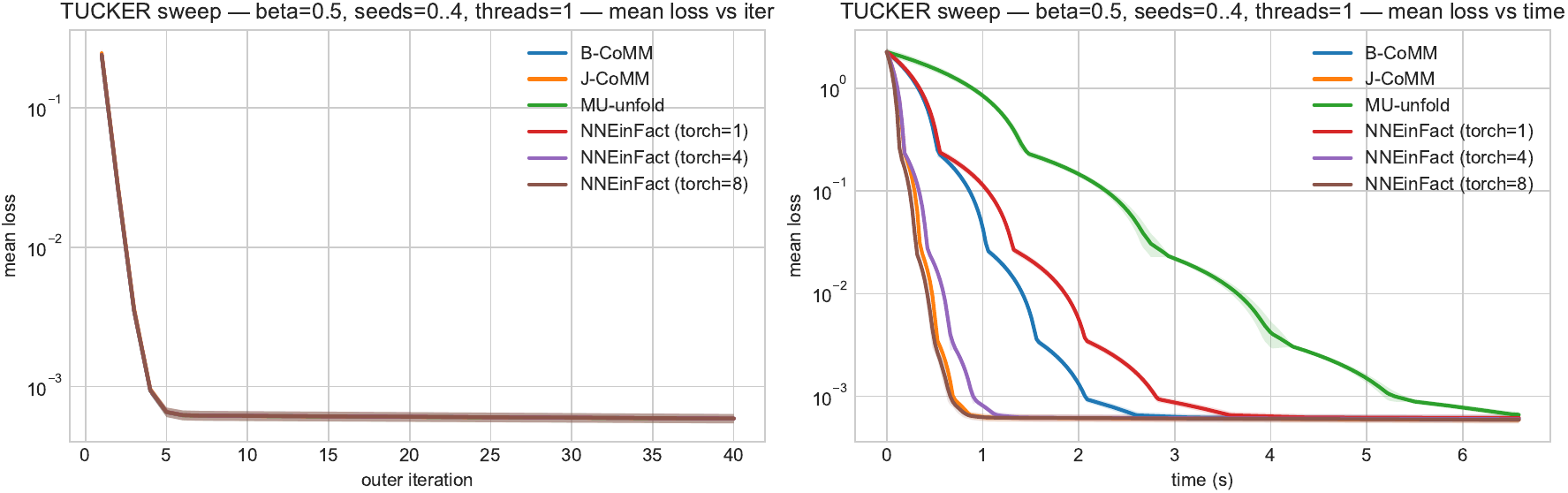}

\vspace{2mm}
\noindent\textbf{Synthetic Tucker, $\beta=1$}\par\vspace{1mm}
\includegraphics[width=0.98\textwidth]{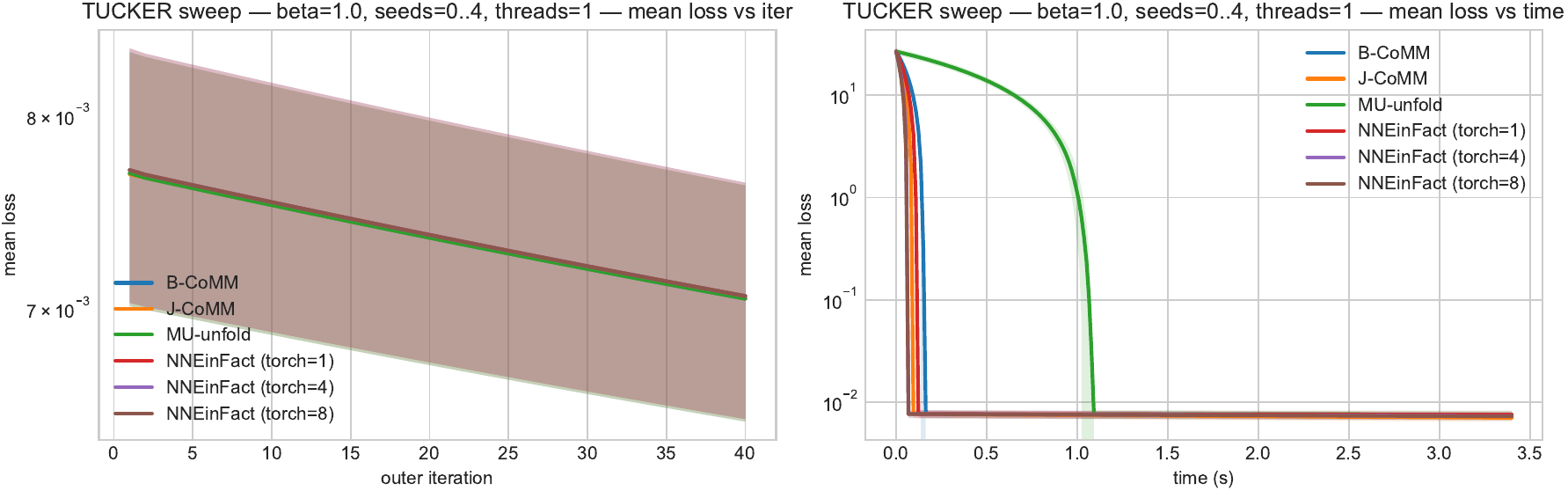}

\vspace{2mm}
\noindent\textbf{Synthetic Tucker, $\beta=1.5$}\par\vspace{1mm}
\includegraphics[width=0.98\textwidth]{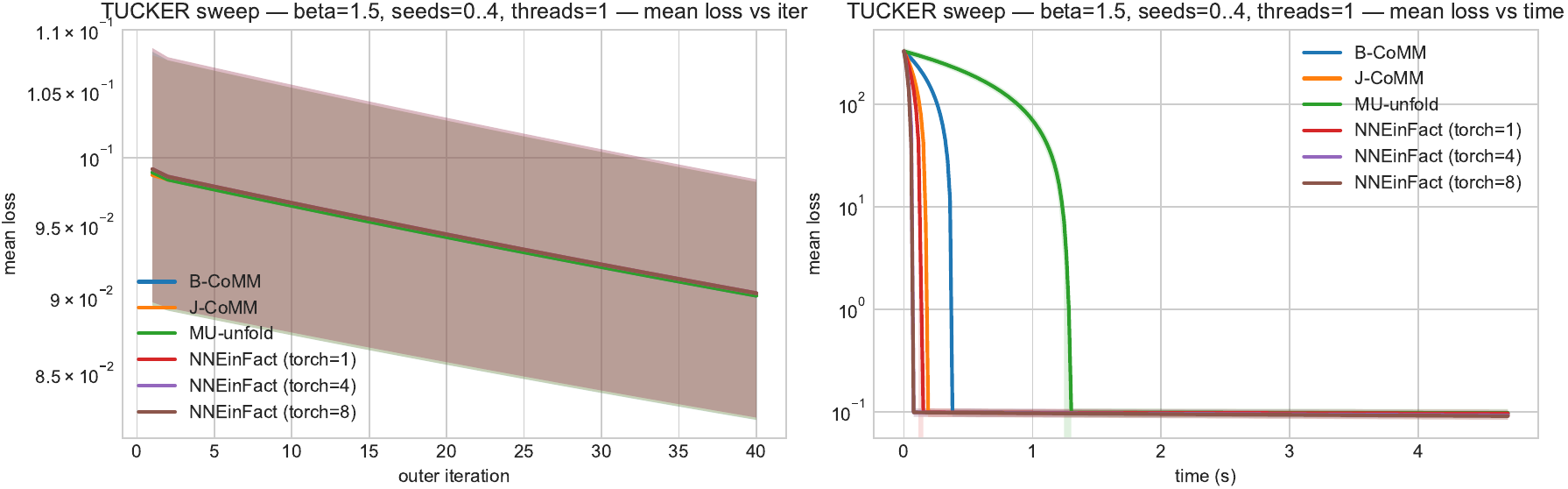}

\caption{Synthetic Tucker benchmark (order $4$, size $80{\times}70{\times}60{\times}50$, multilinear ranks $(10,9,8,7)$).
Each row corresponds to one $\beta\in\{0.5,1,1.5\}$ and shows the mean normalized loss $\bar D_\beta(X,\widehat X)$
versus iteration (left) and wall-clock CPU time (right), averaged over $5$ random seeds (shaded band: variability).
For runtime fairness, NumPy/BLAS is restricted to one thread; NNEinFact is additionally reported for Torch CPU threads
$1/4/8$ (see legend).}
\label{fig:synth_tucker}
\end{figure*}

\paragraph{Observations (synthetic CP and Tucker)}
Figures~\ref{fig:synth_cp}-\ref{fig:synth_tucker} show that, when plotted \emph{per iteration}, all methods exhibit
similar descent profiles: contraction-only implementations preserve the iteration-wise behavior of the underlying
MU/MM updates. The main differences appear in \emph{wall-clock time}.

\smallskip
\noindent\textbf{CP.}
Across $\beta\in\{1/2,1,3/2\}$, J-CoMM consistently provides the best (or near-best) CPU-time performance, with
B-CoMM close behind. In particular, J-CoMM is competitive with the fastest NNEinFact setting (8 threads) and,
for the \emph{same} number of CPU threads, our contraction-only methods (B-CoMM and J-CoMM) are faster across all CP
benchmarks. This confirms that coupling contraction-only updates with \emph{joint majorization} (a fixed surrogate over
a short inner loop) yields substantial reductions in wall-clock time.

\smallskip
\noindent\textbf{Competitor at $\beta=0$.}
For CP, we do not report NNEinFact at $\beta=0$ because in our setting it did not reliably decrease the objective,
whereas our MM updates remain stable for all $\beta\in[0,2)$.

\smallskip
\noindent\textbf{Tucker.}
The same trend holds: CPU-time curves reveal large speedups of contraction-only methods over unfolding-based MU.
NNEinFact is slightly ahead in these Tucker experiments, but J-CoMM remains very close and consistently
outperforms both unfolding-based baselines and the 1-thread competitor, while B-CoMM also yields clear
runtime gains. Overall, the results support the practical benefit of contraction-only formulas and joint majorization
for multilinear $\beta$-divergence objectives.

\subsection{Real data: Uber pickups tensor}
\label{subsec:uber}

\paragraph{Real data: Uber pickups (optimization benchmark)}
We use the \emph{Uber pickups} dataset in the same tensor format as the NNEinFact demo~\cite{hoodschein2026}, namely a nonnegative
$5$-way count tensor $X\in\R_+^{27\times 24\times 7\times 100\times 100}$ whose modes correspond to
(week, hour, day-of-week, latitude index, longitude index).
We fit a nonnegative Tucker model with multilinear ranks $(10,10,5,10,10)$ under the $\beta$-divergence loss.
For this real-data experiment, we restrict attention to optimization performance and report only the normalized
objective value versus (i) outer iteration and (ii) wall-clock CPU time.
Figure~\ref{fig:uber_tucker} reports the results for $\beta\in\{1/2,1,3/2\}$.

\begin{figure*}[ht!]
  \centering

  \includegraphics[width=0.98\textwidth]{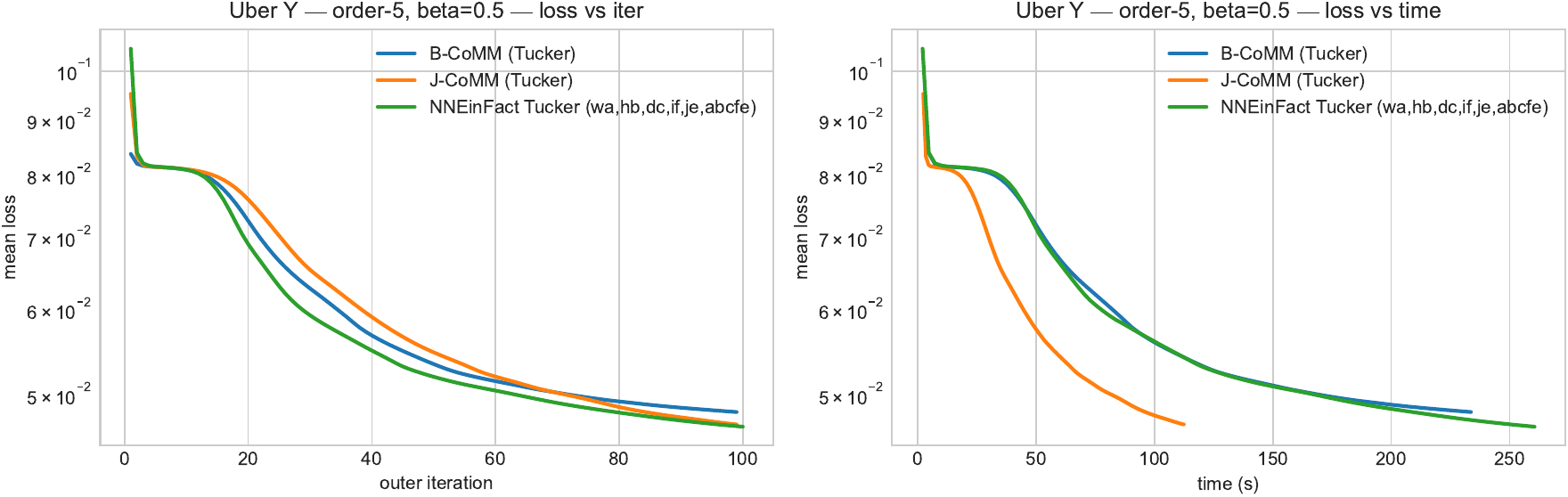}

  \medskip
  \includegraphics[width=0.98\textwidth]{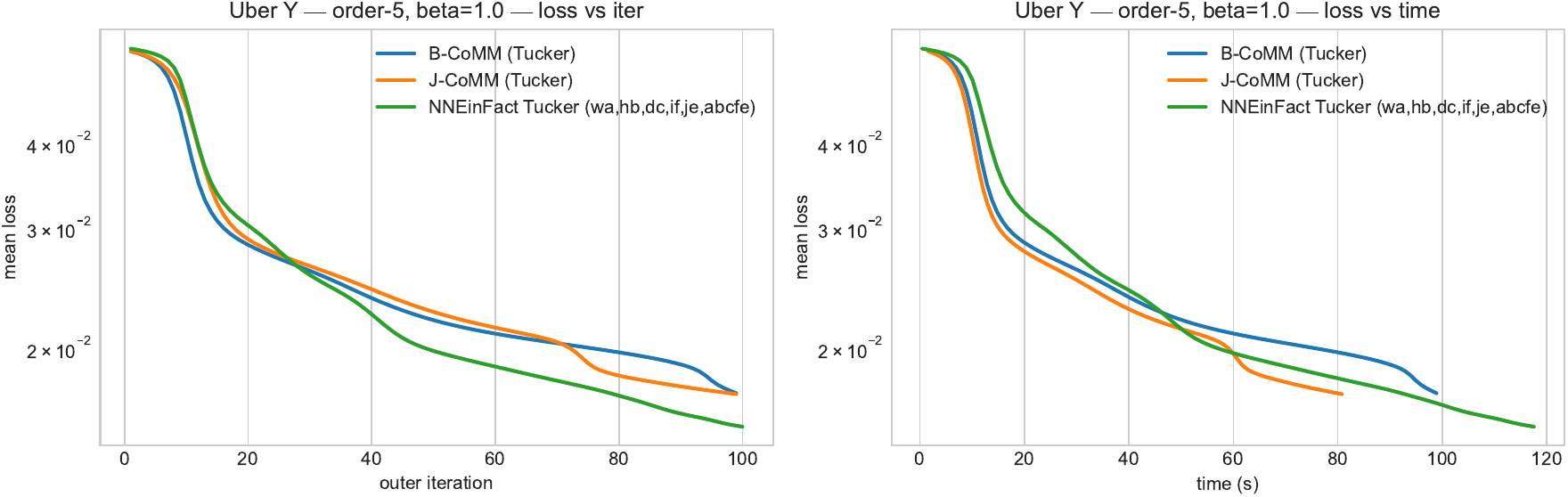}

  \medskip
  \includegraphics[width=0.98\textwidth]{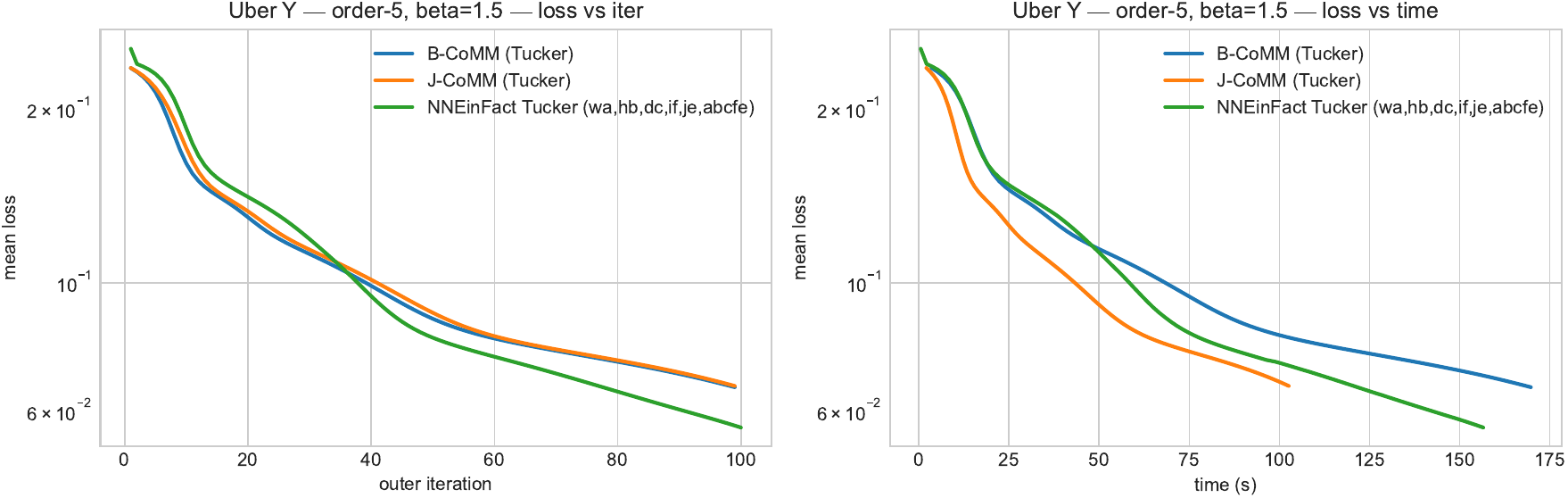}

  \caption{Uber pickups tensor ($27\times 24\times 7\times 100\times 100$): Tucker fit with ranks $(10,10,5,10,10)$.
Each row corresponds to one $\beta\in\{1/2,1,3/2\}$ and shows the normalized objective (mean $\beta$-divergence per entry)
versus outer iteration (left) and wall-clock CPU time (right).
All methods are run in a single-thread CPU configuration (NumPy/BLAS and PyTorch restricted to one thread).}
  \label{fig:uber_tucker}
\end{figure*}

\paragraph{Observations (Uber Tucker benchmark)}
Across all tested values of $\beta$, all methods substantially decrease the normalized $\beta$-divergence.
The main differences appear in wall-clock \emph{time-to-loss}: J-CoMM consistently reaches a given objective
level faster than the other methods, confirming that \emph{reusing a single joint majorizer over a short inner loop}
can yield practical speedups in large-scale settings.
In contrast, NNEinFact is often competitive in \emph{loss vs iteration} and, for $\beta\ge 1$ on this dataset,
can attain the lowest final objective among the compared methods, while remaining slightly slower than J-CoMM in
wall-clock time.
The Uber tensor is highly sparse (about $99.4\%$ zeros), but in these experiments we used dense contractions for all
methods and did not exploit sparsity (e.g., by accumulating numerator terms only over nonzeros), in order to keep the
comparison aligned with the reference implementation and to isolate algorithmic effects from sparse-engineering choices.
Finally, unlike the competitor in its current implementation, our methods handle the full range $\beta\in[0,2)$
reliably; in particular, for the Itakura-Saito case $\beta=0$ we apply a standard small positive floor to $X$ to avoid
an infinite objective caused by zero entries.

\section{Conclusions and Future Work}\label{sec:conclusion}

We studied majorization-minimization methods for nonnegative CP and Tucker decompositions under the entrywise
$\beta$-divergence family, with an emphasis on \emph{unfolding-free} implementations.
Our first contribution is to rewrite classical block-MM multiplicative updates in a \emph{contraction-only} form,
so that all numerators and denominators can be computed directly by tensor contractions (einsum-style operations),
without explicit matricizations or large auxiliary matrices.
Our second and main contribution is a \emph{joint} majorization strategy: at each outer iteration, we build a single
surrogate at a reference point and decrease it through inexpensive inner updates while reusing cached
reference-powered tensors.
We proved tightness of the proposed majorizers and established monotonic decrease of the objective (after each block
update for block-MM and after each outer iteration for joint-MM), which yields convergence of the corresponding sequence
of objective values.
Beyond this, we showed that block-MM fits naturally within the BSUM viewpoint, and we established for J-CoMM a
KL-based iterate-convergence result for one inner sweep per outer iteration under standard regularity assumptions.

Numerical experiments on synthetic tensors and the Uber spatiotemporal count tensor confirm the practical benefits of
the proposed joint-MM scheme.
Across all tested values of $\beta$, the per-iteration progress of the different MM-based methods is broadly comparable,
whereas the main differences appear in wall-clock time.
In particular, for the CP model, we consistently observe significant speedups over unfolding-based baselines and
competitive performance relative to the einsum-factorization competitor; under the same CPU threading budget, our
methods are faster across all reported tests.
These results show that coupling joint majorization with contraction-only multiplicative updates can lead to substantial
runtime savings: by keeping a single surrogate fixed over a short inner loop, the method updates each block through
inexpensive tensor contractions while reusing the reference-powered quantities required by the surrogate.

Several extensions are natural directions for future research:
\begin{itemize}
\item \textbf{Regularized models and additional constraints.}
Extend the contraction-only block and joint MM constructions to regularized nonnegative low-rank approximation
problems (e.g., with scale-invariant regularization and other penalties), building on the framework of
\cite{Cohen_Leplat_2025}.

\item \textbf{Richer multilinear factorizations.}
Generalize the approach beyond CP and Tucker to more expressive constrained tensor models, such as nonnegative
block-term decompositions (BTD) and related structured multilinear formats, where the model remains a sum of
nonnegative multilinear contributions.

\item \textbf{Beyond standard multilinear contractions.}
Investigate how far the majorization principles developed here extend to generic nonnegative einsum models and
contraction graphs (possibly with shared factors and repeated indices), and identify broad classes of models for which
one can obtain tight separable majorizers and closed-form multiplicative updates.

\item \textbf{Sharper convergence theory and accelerated variants.}
Extend the iterate-convergence analysis of J-CoMM beyond the current setting, for instance to multiple inner sweeps,
weaker compactness or positivity assumptions, and broader classes of feasible sets; and analyze practical
acceleration mechanisms (e.g., extrapolation) within the tensor joint-MM setting.
\end{itemize}

\appendix

\section{Proofs for block MM majorizers and multiplicative updates}
\label{app:proof_block_mm}

This appendix provides a complete proof of the surrogate construction and the block multiplicative
updates stated in Theorems~\ref{thm:block_cp} and~\ref{thm:block_tucker}.
The proofs follow the standard MM pattern: (i) build an entrywise upper bound using Jensen's inequality
(and, when $\beta<1$, a convex-concave split with a tangent bound), (ii) show tightness at the current
iterate, (iii) exploit separability in the updated block, and (iv) minimize the resulting 1D functions.

\subsection{Scalar \texorpdfstring{$\beta$}{beta}-divergence as a function of the model value}
Fix $x\ge 0$ and consider $y>0\mapsto d_\beta(x\mid y)$.
For $\beta\neq 0,1$ we rewrite
\begin{equation}
\label{eq:dbeta_split}
d_\beta(x\mid y)
=
\underbrace{\frac{1}{\beta}y^\beta}_{\text{term (I)}}
\;+\;
\underbrace{\frac{x}{1-\beta}y^{\beta-1}}_{\text{term (II)}}
\;+\; \text{constant in $y$}.
\end{equation}

For $\beta\in[0,1)$, term (I) is concave in $y$ while term (II) is convex in $y$.
For $\beta\in[1,2]$, both terms are convex in $y$ (indeed, $y^\beta$ is convex for $\beta\ge 1$,
and for $\beta\in[1,2]$ the map $y^{\beta-1}$ is concave while $\frac{x}{1-\beta}\le 0$, so term (II) is convex),
hence $d_\beta(x\mid y)$ is convex in $y$ and Jensen can be applied directly at the level of $d_\beta$.

\noindent\textbf{Remark (limit cases).}
The cases $\beta=1$ (generalized KL) and $\beta=0$ (Itakura-Saito) are obtained by continuity as limits
$\beta\to 1$ and $\beta\to 0$, respectively; hence expressions involving $1/\beta$ or $1/(\beta-1)$
are understood in this limiting sense.

\subsection{A Jensen majorizer for a convex function of a sum}
Let $\varphi:\R_{++}\to\R$ be convex and let $y=\sum_{\rho} z_\rho$ with $z_\rho\ge 0$.
Fix reference values $\tilde z_\rho\ge 0$ with $\tilde y=\sum_\rho \tilde z_\rho>0$ and define weights
\[
\tilde\lambda_\rho := \frac{\tilde z_\rho}{\tilde y},\qquad \sum_\rho \tilde\lambda_\rho=1.
\]
Then Jensen's inequality gives the standard MM bound
\begin{equation}
\label{eq:jensen_sum}
\varphi\!\Big(\sum_{\rho} z_\rho\Big)
=
\varphi\!\Big(\sum_{\rho}\tilde\lambda_\rho \frac{z_\rho}{\tilde\lambda_\rho}\Big)
\le
\sum_{\rho}\tilde\lambda_\rho\,\varphi\!\Big(\frac{z_\rho}{\tilde\lambda_\rho}\Big),
\end{equation}
and equality holds at $z_\rho=\tilde z_\rho$ for all $\rho$.

\subsection{A tangent upper bound for a concave term}
Let $\psi:\R_{++}\to\R$ be concave and differentiable. Then for any $\tilde y>0$,
\begin{equation}
\label{eq:tangent_concave}
\psi(y)\le \psi(\tilde y)+\psi'(\tilde y)(y-\tilde y),
\end{equation}
with equality at $y=\tilde y$.

\subsection{Entrywise surrogate for the \texorpdfstring{$\beta$}{beta}-divergence when the model is a sum}
Consider a model entry of the form
\[
\widehat{X}_i(\theta)=\sum_{\rho\in\mathcal R} z_{i,\rho}(\theta),\qquad z_{i,\rho}(\theta)\ge 0.
\]
Fix a reference point $\tilde\theta$, set $\tilde X_i=\widehat{X}_i(\tilde\theta)$, and define
$\tilde\lambda_{i,\rho}=z_{i,\rho}(\tilde\theta)/\tilde X_i$.

\paragraph{Case $1\le\beta<2$.}
Since $y\mapsto d_\beta(x\mid y)$ is convex on $y>0$ in this regime, applying
\eqref{eq:jensen_sum} to $\varphi(y)=d_\beta(X_i\mid y)$ yields
\begin{equation}
\label{eq:surrogate_beta_ge_1}
d_\beta\!\Big(X_i\,\Big|\,\sum_\rho z_{i,\rho}(\theta)\Big)
\le
\sum_\rho \tilde\lambda_{i,\rho}\,
d_\beta\!\Big(X_i\,\Big|\,\frac{z_{i,\rho}(\theta)}{\tilde\lambda_{i,\rho}}\Big),
\end{equation}
with equality at $\theta=\tilde\theta$.

\paragraph{Case $0\le\beta<1$.}
Use the split \eqref{eq:dbeta_split}. Apply Jensen \eqref{eq:jensen_sum} to the convex term (II),
and apply the tangent bound \eqref{eq:tangent_concave} to the concave term (I).
This yields an explicit entrywise surrogate of the form
\begin{equation}
\label{eq:surrogate_beta_lt_1}
d_\beta(X_i\mid \widehat{X}_i(\theta))
\le
\sum_{\rho}\tilde\lambda_{i,\rho}\,
\frac{X_i}{1-\beta}\Big(\frac{z_{i,\rho}(\theta)}{\tilde\lambda_{i,\rho}}\Big)^{\beta-1}
\;+\;
\Big(\frac{1}{\beta}\tilde X_i^\beta + \tilde X_i^{\beta-1}(\widehat{X}_i(\theta)-\tilde X_i)\Big)
\;+\; \text{const},
\end{equation}
again tight at $\theta=\tilde\theta$.

Summing \eqref{eq:surrogate_beta_ge_1} or \eqref{eq:surrogate_beta_lt_1} over all indices $i$ produces a global
surrogate for $D_\beta(\X,\widehat{\X}(\theta))$.

\subsection{Separable minimization and the multiplicative update}

We now explain why minimizing the surrogate w.r.t.\ one block yields the multiplicative update.

Fix a block variable (an entry of a factor matrix or the core) denoted $u\ge\eps$, and collect all
surrogate terms that depend on $u$.
In both CP and Tucker, each $z_{i,\rho}(\theta)$ is multilinear, hence \emph{linear} in $u$ when the other
blocks are fixed. Consequently, the surrogate becomes separable in the entries of the updated block.

\paragraph{Resulting 1D forms.}
After collecting constants, the scalar surrogate in one variable $u\ge\eps$ takes the form
\[
\text{if } 1<\beta<2:\quad g(u)=a\,u^\beta - b\,u^{\beta-1} + \text{const},\qquad a>0,\;b\ge 0,
\]
with the case $\beta=1$ (KL) obtained by continuity and yielding
\[
\text{if } \beta=1:\quad g(u)=a\,u - b\log u + \text{const},\qquad a>0,\;b\ge 0,
\]
and for $0\le\beta<1$,
\[
\text{if } 0\le\beta<1:\quad g(u)=a\,u + b\,u^{\beta-1} + \text{const},\qquad a>0,\;b>0.
\]
If $b=0$, the unconstrained minimizer is at $u=0$; on the $\varepsilon$-constrained set the minimizer is attained at the lower bound $u=\varepsilon$.

Each function is convex on $u>0$ and is strictly convex whenever $b>0$; in all cases, on the
$\varepsilon$-constrained set $u\ge\varepsilon$ the minimizer is unique (if $b=0$, it is attained at $u=\varepsilon$).

Setting $g'(u)=0$ yields a closed-form minimizer.
In particular, for $1\le\beta<2$ the minimizer depends linearly on the ratio $b/a$ (hence exponent $1$),
and for $0\le\beta<1$ the minimizer yields exponent $1/(2-\beta)$.
This exactly matches the exponent $\gamma(\beta)$ of Remark~\ref{rem:gamma}.
Here $\mathrm{Num}$ and $\mathrm{Den}$ are the contractions that appear in the main text
(e.g., $\mathrm{Num}^{(n)}$, $\mathrm{Den}^{(n)}$), obtained by collecting the coefficients $a$ and $b$.

\subsection{Application to CP: proof of Theorem~\ref{thm:block_cp}}
For CP, each model entry is a sum over rank components:
\[
\widehat{X}_i = \sum_{r=1}^R z_{i,r},\qquad z_{i,r}:=\prod_{n=1}^N A^{(n)}_{i_n r}.
\]
Fix a mode $n$ and update $A^{(n)}$ with other factors fixed. Each $z_{i,r}$ is linear in
$A^{(n)}_{i_n r}$:
\[
z_{i,r} = A^{(n)}_{i_n r}\, s^{(n)}_r(i_{-n}),\qquad
s^{(n)}_{r}(i_{-n})=\prod_{m\neq n} A^{(m)}_{i_m r}.
\]
Applying the entrywise surrogate above and collecting terms depending on a single entry
$A^{(n)}_{i_n r}$ yields the 1D surrogate form described previously, with coefficients
\[
\mathrm{Num}^{(n)}_{i_n r}=\sum_{i_{-n}} \mathcal{P}_{i}\,s^{(n)}_{r}(i_{-n}),\qquad
\mathrm{Den}^{(n)}_{i_n r}=\sum_{i_{-n}} \mathcal{Q}_{i}\,s^{(n)}_{r}(i_{-n}),
\]
which leads to the multiplicative update of Theorem~\ref{thm:block_cp}.
Monotonic descent follows from the MM property (surrogate minimization decreases the objective).

\subsection{Application to Tucker: proof of Theorem~\ref{thm:block_tucker}}
For Tucker, each entry is
\[
\widehat{X}_i = \sum_{j_1,\dots,j_N} z_{i,j},\qquad
z_{i,j} := G_{j_1\dots j_N}\prod_{n=1}^N A^{(n)}_{i_n j_n}.
\]
Fixing all blocks except one makes $z_{i,j}$ linear in any single updated entry, so the same
surrogate construction applies.
Carrying out the collection of coefficients yields exactly the contractions shown in the main text:
the core contractions $\mathcal{P}_{\mathrm{core}}$, $\mathcal{Q}_{\mathrm{core}}$ for updating $G$, and the
factor contractions $\mathrm{Num}^{(n)}$, $\mathrm{Den}^{(n)}$ for updating $A^{(n)}$.
The multiplicative updates and monotonic descent follow identically.

\section{Proof details for the joint majorizer and the inner updates}
\label{app:proof_joint_mm}

This appendix provides the missing proof details for the joint majorization-minimization (joint-MM)
construction used in Section~\ref{sec:jmm_multilinear}.
We focus on two statements:

\begin{itemize}
\item the construction of a \emph{single} surrogate $G(\Theta\mid \widetilde\Theta)$ that majorizes the objective
for \emph{all} variables $\Theta$ jointly, and is tight at the reference point $\widetilde\Theta$;
\item why the inner multiplicative updates decrease that fixed surrogate.
\end{itemize}

Throughout, we work under Assumption~\ref{ass:eps} so that all quantities are well-defined.

\subsection{A generic entrywise setup}
Fix one tensor entry index $i=(i_1,\dots,i_N)$.
Assume that the model prediction can be written as a sum of nonnegative contributions
\begin{equation}
\label{eq:app_sum_model}
\widehat{\X}_i(\Theta) = \sum_{\rho\in\mathcal{R}} z_{i,\rho}(\Theta),
\qquad z_{i,\rho}(\Theta)\ge 0,
\end{equation}
where $\rho$ indexes components.

\paragraph{CP.}
$\mathcal{R}=\{1,\dots,R\}$ and $z_{i,r}(\Theta)=\prod_{n=1}^N A^{(n)}_{i_n r}$.

\paragraph{Tucker.}
$\mathcal{R}=\{1,\dots,J_1\}\times\cdots\times\{1,\dots,J_N\}$ and
$z_{i,j}(\Theta)=G_{j_1\dots j_N}\prod_{n=1}^N A^{(n)}_{i_n j_n}$.

Let $\widetilde\Theta$ be a reference point and define $\widetilde{\Xhat}_i:=\widehat{\X}_i(\widetilde\Theta)>0$.
Define the reference weights
\begin{equation}
\label{eq:app_lambda}
\widetilde\lambda_{i,\rho}:=\frac{z_{i,\rho}(\widetilde\Theta)}{\widetilde{\Xhat}_i},
\qquad
\sum_{\rho\in\mathcal{R}}\widetilde\lambda_{i,\rho}=1.
\end{equation}
These weights depend only on $\widetilde\Theta$, hence they are constant during the inner loop.

\subsection{Entrywise upper bound for \texorpdfstring{$\beta\in[1,2)$}{beta in [1,2)}}
For fixed $x\ge 0$, when $\beta\in[1,2)$ the map $y\mapsto d_\beta(x\mid y)$ is convex on $y>0$.
Using \eqref{eq:app_sum_model} and \eqref{eq:app_lambda}, write
\[
\widehat{\X}_i(\Theta)=\sum_{\rho}\widetilde\lambda_{i,\rho}\,
\frac{z_{i,\rho}(\Theta)}{\widetilde\lambda_{i,\rho}}.
\]
Jensen's inequality gives the entrywise majorization
\begin{equation}
\label{eq:app_jensen_beta_ge_1}
d_\beta\!\Big(\X_i \,\Big|\, \widehat{\X}_i(\Theta)\Big)
\le
\sum_{\rho\in\mathcal{R}}\widetilde\lambda_{i,\rho}\,
d_\beta\!\Big(\X_i \,\Big|\, \frac{z_{i,\rho}(\Theta)}{\widetilde\lambda_{i,\rho}}\Big).
\end{equation}
Moreover, equality holds at $\Theta=\widetilde\Theta$ because then
$z_{i,\rho}(\Theta)/\widetilde\lambda_{i,\rho}=\widetilde{\Xhat}_i$ for all $\rho$.

\subsection{Entrywise upper bound for \texorpdfstring{$\beta\in[0,1)$}{beta in [0,1)}}
When $\beta\in[0,1)$, the standard MM construction uses a convex--concave split in the second argument.
One can write $d_\beta(x\mid y)$ as a sum of a convex part in $y$ and a concave part in $y$.
The convex part is majorized by Jensen in the same way as \eqref{eq:app_jensen_beta_ge_1},
and the concave part is upper bounded by its first-order Taylor expansion at $y=\widetilde{\Xhat}_i$.
This again yields an entrywise bound of the form
\begin{equation}
\label{eq:app_beta_lt_1_bound}
d_\beta(\X_i\mid \widehat{\X}_i(\Theta)) \le G_i(\Theta\mid\widetilde\Theta),
\qquad
G_i(\widetilde\Theta\mid\widetilde\Theta)=d_\beta(\X_i\mid \widetilde{\Xhat}_i).
\end{equation}
The important point for the algorithm is that $G_i(\cdot\mid\widetilde\Theta)$ is a sum over $\rho$
whose dependence on $z_{i,\rho}(\Theta)$ can be separated using the fixed weights \eqref{eq:app_lambda}.

\paragraph{Remark.}
The resulting formulas coincide with the classical $\beta$-NMF MM constructions, with the only change that
the index $\rho$ may represent CP components ($r$) or Tucker multi-indices ($j$).

\subsection{From entrywise bounds to a global joint surrogate}
Summing \eqref{eq:app_jensen_beta_ge_1} (for $\beta\in[1,2)$) or \eqref{eq:app_beta_lt_1_bound}
(for $\beta\in[0,1)$) over all tensor entries $i$ yields a global surrogate
\begin{equation}
\label{eq:app_global_surrogate}
F(\Theta)=D_\beta(\X,\widehat\X(\Theta))
\le
G(\Theta\mid\widetilde\Theta):=\sum_i G_i(\Theta\mid\widetilde\Theta),
\qquad
G(\widetilde\Theta\mid\widetilde\Theta)=F(\widetilde\Theta).
\end{equation}
This is the sense in which the surrogate is \emph{joint}:
\eqref{eq:app_global_surrogate} holds for all variables $\Theta$ simultaneously.

\subsection{Explicit scalar subproblems and exact block minimizers of \texorpdfstring{$G(\cdot\mid\widetilde\Theta)$}{G(.|Theta-tilde)}}
\label{app:jmm_scalar}

We make explicit the scalar functions $g_{i_nr}$ whose minimizers yield the inner multiplicative updates.
This clarifies an important nuance: $G(\Theta\mid\widetilde\Theta)$ is a \emph{single} joint surrogate valid for all
variables, but it is not jointly separable across all blocks at once. Instead, it is \emph{blockwise} entrywise
separable: when all blocks except one are fixed, $G(\cdot\mid\widetilde\Theta)$ decomposes as a sum of independent
one-dimensional convex functions over the entries of the active block.

\paragraph{Reference tensors and transforms.}
Fix a reference $\widetilde\Theta$ and denote $\widetilde{\Xhat}=\widehat\X(\widetilde\Theta)$.
Define the reference-powered tensors
\[
\widetilde{\mathcal P}:=\X\odot \widetilde{\Xhat}^{\beta-2},
\qquad
\widetilde{\mathcal Q}:=\widetilde{\Xhat}^{\beta-1}.
\]
For a nonnegative variable $Z$ with reference $\widetilde Z$, define the entrywise transforms
\[
\chi_{1,\beta}(Z,\widetilde Z):=\widetilde Z^{\,2-\beta}\odot Z^{\,\beta-1}
=\widetilde Z\odot \Big(\frac{Z}{\widetilde Z}\Big)^{\beta-1},
\]
\[
\chi_{2,\beta}(Z,\widetilde Z):=
\begin{cases}
Z, & 0\le \beta<1,\\[1mm]
Z^{\,\beta}\odot \widetilde Z^{-(\beta-1)}
=\widetilde Z\odot \Big(\frac{Z}{\widetilde Z}\Big)^{\beta}, & 1\le \beta<2.
\end{cases}
\]

\begin{lemma}[CP: explicit scalar form and unique block minimizer of the joint surrogate]
\label{lem:jmm_cp_scalar}
Consider the CP model $\widehat X_i(\Theta)=\sum_{r=1}^R \prod_{n=1}^N A^{(n)}_{i_n r}$.
Fix a reference $\widetilde\Theta=\{\widetilde A^{(n)}\}$ and the associated joint surrogate
$G(\Theta\mid\widetilde\Theta)$ defined by the entrywise bounds in
\eqref{eq:app_jensen_beta_ge_1} (for $1\le\beta<2$) and \eqref{eq:app_beta_lt_1_bound} (for $0\le\beta<1$),
summed over all indices $i$.
Fix an inner iterate $\Theta$ and update one factor matrix $A^{(n)}$ while keeping all other factors fixed.

Define the entrywise \emph{ratio variables}
\[
U^{(n)}:=A^{(n)}\oslash \widetilde A^{(n)}\qquad(\text{entrywise division}),\quad
u:=U^{(n)}_{i_n r}=\frac{A^{(n)}_{i_n r}}{\widetilde A^{(n)}_{i_n r}}.
\]
Then, with all other blocks fixed, the joint surrogate decomposes as
\[
G(\Theta\mid\widetilde\Theta)
=
\mathrm{const}
+
\sum_{i_n=1}^{I_n}\sum_{r=1}^R g_{i_n r}(u_{i_n r}),
\]
where each scalar function $g_{i_n r}$ is strictly convex on $(0,\infty)$ and can be written explicitly as:

\smallskip
\noindent
(i) if $1<\beta<2$:
\[
g_{i_n r}(u)
=
\frac{\mathrm{Den}^{(n)}_{\mathrm J}(i_n,r)}{\beta}\,u^{\beta}
-
\frac{\mathrm{Num}^{(n)}_{\mathrm J}(i_n,r)}{\beta-1}\,u^{\beta-1}
+\mathrm{const},
\]

\noindent
(ii) if $\beta=1$ (limit case):
\[
g_{i_n r}(u)
=
\mathrm{Den}^{(n)}_{\mathrm J}(i_n,r)\,u
-
\mathrm{Num}^{(n)}_{\mathrm J}(i_n,r)\,\log u
+\mathrm{const},
\]

\noindent
(iii) if $0\le\beta<1$:
\[
g_{i_n r}(u)
=
\mathrm{Den}^{(n)}_{\mathrm J}(i_n,r)\,u
+
\frac{\mathrm{Num}^{(n)}_{\mathrm J}(i_n,r)}{1-\beta}\,u^{\beta-1}
+\mathrm{const}.
\]

The coefficients are the contraction-only quantities
\[
\mathrm{Num}^{(n)}_{\mathrm J}(i_n,r)
:=
\sum_{i_{-n}}
\widetilde{\mathcal P}_{i}\,
\prod_{m\neq n}\chi_{1,\beta}\!\left(A^{(m)}_{i_m r},\widetilde A^{(m)}_{i_m r}\right),
\]
\[
\mathrm{Den}^{(n)}_{\mathrm J}(i_n,r)
:=
\sum_{i_{-n}}
\widetilde{\mathcal Q}_{i}\,
\prod_{m\neq n}\chi_{2,\beta}\!\left(A^{(m)}_{i_m r},\widetilde A^{(m)}_{i_m r}\right).
\]

Consequently, $g_{i_n r}$ has a unique minimizer on the $\varepsilon$-constrained feasible set
(equivalently, for $u\ge \varepsilon/\widetilde A^{(n)}_{i_n r}$); when $\mathrm{Num}^{(n)}_{\mathrm J}(i_n,r)=0$
the minimizer is attained at the lower bound, and is given by

\[
u^\star=
\begin{cases}
\displaystyle \frac{\mathrm{Num}^{(n)}_{\mathrm J}(i_n,r)}{\mathrm{Den}^{(n)}_{\mathrm J}(i_n,r)},
& 1\le \beta<2,\\[3mm]
\displaystyle \left(\frac{\mathrm{Num}^{(n)}_{\mathrm J}(i_n,r)}{\mathrm{Den}^{(n)}_{\mathrm J}(i_n,r)}\right)^{\!\frac{1}{2-\beta}},
& 0\le \beta<1,
\end{cases}
\]
and the corresponding minimizer in the original variables is exactly the inner multiplicative update
\[
A^{(n)}_{i_n r}\;\leftarrow\;\widetilde A^{(n)}_{i_n r}\,(u^\star)
=
\widetilde A^{(n)}_{i_n r}\,
\left(\frac{\mathrm{Num}^{(n)}_{\mathrm J}(i_n,r)}{\mathrm{Den}^{(n)}_{\mathrm J}(i_n,r)}\right)^{\gamma(\beta)}.
\]
In particular, each inner block update is the unique minimizer of $G(\cdot\mid\widetilde\Theta)$ with respect to that
block (holding all other blocks fixed).
\end{lemma}

\begin{proof}
We give the CP proof; the Tucker proof is identical in structure and is stated as a corollary below.

\paragraph{Step 1: reparametrize contributions by ratios.}
For CP, define for each entry $i$ and component $r$
\[
\widetilde z_{i,r}:=\prod_{m=1}^N \widetilde A^{(m)}_{i_m r},
\qquad
U^{(m)}_{i_m r}:=\frac{A^{(m)}_{i_m r}}{\widetilde A^{(m)}_{i_m r}},
\qquad
v_{i,r}:=\prod_{m=1}^N U^{(m)}_{i_m r}.
\]
Then $z_{i,r}(\Theta)=\prod_m A^{(m)}_{i_m r}=\widetilde z_{i,r}\,v_{i,r}$ and
$\widetilde\lambda_{i,r}=\widetilde z_{i,r}/\widetilde{\Xhat}_i$ as in \eqref{eq:app_lambda}.

\paragraph{Step 2: explicit entrywise surrogate in terms of $v_{i,r}$.}

For $1<\beta<2$, start from Jensen \eqref{eq:app_jensen_beta_ge_1}:
\[
G_i(\Theta\mid\widetilde\Theta)
=
\sum_{r=1}^R
\widetilde\lambda_{i,r}\,
d_\beta\!\Big(\X_i \,\Big|\, \frac{z_{i,r}(\Theta)}{\widetilde\lambda_{i,r}}\Big).
\]
Using the closed form of $d_\beta$ and discarding terms independent of $\Theta$, we obtain
\[
G_i(\Theta\mid\widetilde\Theta)
=
\sum_{r=1}^R
\left[
\frac{1}{\beta}\,
\widetilde\lambda_{i,r}^{\,1-\beta}\,z_{i,r}(\Theta)^{\beta}
-
\frac{\X_i}{\beta-1}\,
\widetilde\lambda_{i,r}^{\,2-\beta}\,z_{i,r}(\Theta)^{\beta-1}
\right]
+\mathrm{const}.
\]
Now substitute
\[
z_{i,r}(\Theta)=\widetilde z_{i,r}\,v_{i,r},
\qquad
\widetilde\lambda_{i,r}=\frac{\widetilde z_{i,r}}{\widetilde{\Xhat}_i}.
\]
Then
\[
\widetilde\lambda_{i,r}^{\,1-\beta}\,z_{i,r}(\Theta)^\beta
=
\left(\frac{\widetilde z_{i,r}}{\widetilde{\Xhat}_i}\right)^{1-\beta}
(\widetilde z_{i,r}v_{i,r})^\beta
=
\widetilde z_{i,r}\,\widetilde{\Xhat}_i^{\beta-1}\,v_{i,r}^{\beta}
=
\widetilde z_{i,r}\,\widetilde{\mathcal Q}_i\,v_{i,r}^{\beta},
\]
and
\[
\widetilde\lambda_{i,r}^{\,2-\beta}\,z_{i,r}(\Theta)^{\beta-1}
=
\left(\frac{\widetilde z_{i,r}}{\widetilde{\Xhat}_i}\right)^{2-\beta}
(\widetilde z_{i,r}v_{i,r})^{\beta-1}
=
\widetilde z_{i,r}\,\widetilde{\Xhat}_i^{\beta-2}\,v_{i,r}^{\beta-1}.
\]
Therefore,
\[
\X_i\,\widetilde\lambda_{i,r}^{\,2-\beta}\,z_{i,r}(\Theta)^{\beta-1}
=
\widetilde z_{i,r}\,(\X_i\widetilde{\Xhat}_i^{\beta-2})\,v_{i,r}^{\beta-1}
=
\widetilde z_{i,r}\,\widetilde{\mathcal P}_i\,v_{i,r}^{\beta-1}.
\]
Hence, up to constants independent of $\Theta$,
\[
G_i(\Theta\mid\widetilde\Theta)
=
\sum_{r=1}^R
\left[
\frac{\widetilde{\mathcal Q}_i}{\beta}\,\widetilde z_{i,r}\,v_{i,r}^{\beta}
-
\frac{\widetilde{\mathcal P}_i}{\beta-1}\,\widetilde z_{i,r}\,v_{i,r}^{\beta-1}
\right]
+\mathrm{const}.
\]

The limit case $\beta=1$ follows by continuity and gives a term in $v_{i,r}$ and $\log v_{i,r}$.
For $0\le\beta<1$, the convex-concave construction \eqref{eq:app_beta_lt_1_bound} yields the same structural outcome:
up to constants independent of $\Theta$,
\[
G_i(\Theta\mid\widetilde\Theta)
=
\sum_{r=1}^R
\left[
\widetilde{\mathcal Q}_i\,\widetilde z_{i,r}\,v_{i,r}
+
\frac{\widetilde{\mathcal P}_i}{1-\beta}\,\widetilde z_{i,r}\,v_{i,r}^{\beta-1}
\right]
+\mathrm{const}.
\]
(These are the classical joint-MM scalar forms; see also the discussion in the matrix case.)

\paragraph{Step 3: isolate one block and obtain entrywise separability.}
Fix a mode $n$ and hold all factors $A^{(m)}$ for $m\neq n$ fixed. Then for each $(i_n,r)$,
\[
v_{i,r} = U^{(n)}_{i_n r}\cdot c_{i_{-n},r},
\qquad
c_{i_{-n},r}:=\prod_{m\neq n} U^{(m)}_{i_m r}
\quad\text{(constant during the update of $A^{(n)}$).}
\]
Also $\widetilde z_{i,r}=\widetilde A^{(n)}_{i_n r}\,\widetilde s^{(n)}_r(i_{-n})$ with
$\widetilde s^{(n)}_r(i_{-n})=\prod_{m\neq n}\widetilde A^{(m)}_{i_m r}$.
Hence, for fixed $(i_n,r)$, all dependence on $A^{(n)}_{i_n r}$ is through
$u:=U^{(n)}_{i_n r}=A^{(n)}_{i_n r}/\widetilde A^{(n)}_{i_n r}$, and summing over all $i$ yields a decomposition
\[
G(\Theta\mid\widetilde\Theta)
=\mathrm{const}+\sum_{i_n,r} g_{i_n r}(u_{i_n r}),
\]
with $g_{i_n r}$ obtained by collecting the terms involving this $u$.

Using
\[
\chi_{1,\beta}(A^{(m)}_{i_m r},\widetilde A^{(m)}_{i_m r})
=
\widetilde A^{(m)}_{i_m r}\,(U^{(m)}_{i_m r})^{\beta-1},
\qquad
\chi_{2,\beta}(A^{(m)}_{i_m r},\widetilde A^{(m)}_{i_m r})
=
\widetilde A^{(m)}_{i_m r}\,(U^{(m)}_{i_m r})^{\beta}
\ \ (1\le\beta<2),
\]
we see that the coefficients multiplying $u^\beta$ and $u^{\beta-1}$ (or $u$ and $u^{\beta-1}$ when $\beta<1$)
are exactly the contraction sums $\mathrm{Den}^{(n)}_{\mathrm J}(i_n,r)$ and
$\mathrm{Num}^{(n)}_{\mathrm J}(i_n,r)$ stated in the lemma.
This yields the explicit scalar forms for $g_{i_n r}$.

\paragraph{Step 4: strict convexity and closed-form minimizer.}
For $1<\beta<2$, $g_{i_n r}(u)$ has the form
$\frac{\mathrm{Den}}{\beta}u^\beta - \frac{\mathrm{Num}}{\beta-1}u^{\beta-1}+\mathrm{const}$
with $\mathrm{Den}>0$ and $\mathrm{Num}\ge 0$, hence
\[
g''(u)=\mathrm{Den}(\beta-1)u^{\beta-2}+\mathrm{Num}(2-\beta)u^{\beta-3}>0,\quad u>0,
\]

so $g$ is strictly convex. If $\mathrm{Num}^{(n)}_{\mathrm J}(i_n,r)>0$, it has a unique minimizer in $(0,\infty)$,
and solving $g'(u)=0$ gives $u^\star=\mathrm{Num}/\mathrm{Den}$.

The cases $\beta=1$ and $0\le\beta<1$ follow similarly (by limit or direct differentiation), yielding the stated
minimizers and the multiplicative update in $A^{(n)}$.

If $\mathrm{Num}^{(n)}_{\mathrm J}(i_n,r)=0$, then the unconstrained minimizer corresponds to $u^\star=0$;
on the $\varepsilon$-constrained set the unique minimizer is attained at the lower bound
$u=\varepsilon/\widetilde A^{(n)}_{i_n r}$, i.e., $A^{(n)}_{i_n r}=\varepsilon$.

\end{proof}

\begin{corollary}[Tucker: explicit scalar forms for core and factor updates]
\label{cor:jmm_tucker_scalar}
Fix a Tucker reference $\widetilde\Theta=\{\widetilde G,\widetilde A^{(1)},\dots,\widetilde A^{(N)}\}$ and the joint
surrogate $G(\Theta\mid\widetilde\Theta)$.
When updating one block (either a factor $A^{(n)}$ or the core $G$) with all other blocks fixed, the surrogate
decomposes entrywise over the active block into strictly convex scalar functions of the ratio variables
$A^{(n)}\oslash \widetilde A^{(n)}$ or $G\oslash\widetilde G$.
The scalar functions have exactly the same forms as in Lemma~\ref{lem:jmm_cp_scalar}, with coefficients given by the
explicit contraction formulas \eqref{eq:jmm_core_num}-\eqref{eq:jmm_core_den} (core) and
\eqref{eq:jmm_factor_num}-\eqref{eq:jmm_factor_den} (factors).
In particular, each inner Tucker core/factor update is the unique minimizer of $G(\cdot\mid\widetilde\Theta)$ with
respect to that block (holding all other blocks fixed).
\end{corollary}

\subsection{Inner decrease and monotonicity of the objective}\label{subsec:Inner_Decrease_and_monot}

Each inner update decreases the fixed surrogate $G(\cdot\mid\widetilde\Theta)$ with respect to the selected block,
which establishes the decrease property used in Proposition~\ref{prop:inner_decrease}.

Consequently, after $L$ inner updates we have
\[
G(\Theta^{(L)}\mid\widetilde\Theta)\le G(\Theta^{(0)}\mid\widetilde\Theta)
=G(\widetilde\Theta\mid\widetilde\Theta).
\]
Using the majorization property $F(\Theta)\le G(\Theta\mid\widetilde\Theta)$ then yields the outer-iteration decrease
\[
F(\Theta^{(L)}) \le G(\Theta^{(L)}\mid\widetilde\Theta)\le G(\widetilde\Theta\mid\widetilde\Theta)=F(\widetilde\Theta),
\]
which is exactly the statement of Theorem~\ref{thm:jmm_monotone}.

\section{Explicit Tucker quantities}
\label{app:tucker_explicit}

This appendix expands all Tucker quantities used in the paper, with explicit indices.
We give formulas for both the block-MM scheme (Section~\ref{sec:block_mm}) and the joint-MM scheme (Section~\ref{sec:jmm_multilinear}).
These expansions are intended to be directly implementable.

\subsection{Tucker model and basic notation}
Let $\X\in\R_+^{I_1\times\cdots\times I_N}$.
Let $\G\in\R_+^{J_1\times\cdots\times J_N}$ and $A^{(n)}\in\R_+^{I_n\times J_n}$.
The Tucker reconstruction is
\begin{equation}
\label{eq:tucker_entry}
\Xhat_{i_1\dots i_N}
=
\sum_{j_1=1}^{J_1}\cdots\sum_{j_N=1}^{J_N}
\G_{j_1\dots j_N}\prod_{n=1}^N A^{(n)}_{i_n j_n}.
\end{equation}

Throughout, $i=(i_1,\dots,i_N)$ and $j=(j_1,\dots,j_N)$.

\subsection{Block-MM quantities for Tucker}
Given a current reconstruction $\Xhat$, define (entrywise)
\[
\mathcal{P}_i := \X_i\,\Xhat_i^{\beta-2},
\qquad
\mathcal{Q}_i := \Xhat_i^{\beta-1}.
\]

\subsubsection{Core update, explicit formula}
The block-MM core update uses
\begin{equation}
\label{eq:block_core_num_den}
\mathcal{P}_{\mathrm{core}}(j_1,\dots,j_N)
=
\sum_{i_1=1}^{I_1}\cdots\sum_{i_N=1}^{I_N}
\mathcal{P}_{i_1\dots i_N}\prod_{n=1}^N A^{(n)}_{i_n j_n},
\end{equation}
\begin{equation}
\label{eq:block_core_den}
\mathcal{Q}_{\mathrm{core}}(j_1,\dots,j_N)
=
\sum_{i_1=1}^{I_1}\cdots\sum_{i_N=1}^{I_N}
\mathcal{Q}_{i_1\dots i_N}\prod_{n=1}^N A^{(n)}_{i_n j_n}.
\end{equation}
Then
\[
\G \leftarrow \G \od \left(\frac{\mathcal{P}_{\mathrm{core}}}{\mathcal{Q}_{\mathrm{core}}}\right)^{\gamma(\beta)}.
\]

\subsubsection{Factor update, explicit formula}
Fix a mode $n$.
Define the partial contraction tensor
\begin{equation}
\label{eq:block_Bn}
\mathcal{B}^{(n)}_{j_n,\,i_{-n}}
:=
\sum_{j_1=1}^{J_1}\cdots\sum_{j_{n-1}=1}^{J_{n-1}}
\sum_{j_{n+1}=1}^{J_{n+1}}\cdots\sum_{j_N=1}^{J_N}
\G_{j_1\dots j_N}\prod_{m\neq n} A^{(m)}_{i_m j_m},
\end{equation}
where $i_{-n}=(i_1,\dots,i_{n-1},i_{n+1},\dots,i_N)$.
Then the numerator and denominator matrices for the factor update are
\begin{equation}
\label{eq:block_factor_num}
\mathrm{Num}^{(n)}_{i_n j_n}
=
\sum_{i_{-n}} \mathcal{P}_{i_1\dots i_N}\,\mathcal{B}^{(n)}_{j_n,\,i_{-n}},
\end{equation}
\begin{equation}
\label{eq:block_factor_den}
\mathrm{Den}^{(n)}_{i_n j_n}
=
\sum_{i_{-n}} \mathcal{Q}_{i_1\dots i_N}\,\mathcal{B}^{(n)}_{j_n,\,i_{-n}}.
\end{equation}
The block-MM factor update is
\[
A^{(n)} \leftarrow A^{(n)} \od \left(\frac{\mathrm{Num}^{(n)}}{\mathrm{Den}^{(n)}}\right)^{\gamma(\beta)}.
\]

\subsection{Joint-MM quantities for Tucker}
We now give fully explicit formulas for the inner updates in the joint-MM scheme.

\subsubsection{Reference tensors}
At an outer iteration, fix a reference point
\[
\widetilde{\Theta}=\{\widetilde{\G},\widetilde A^{(1)},\dots,\widetilde A^{(N)}\},
\qquad
\widetilde{\Xhat}=\widehat{\X}(\widetilde{\Theta}).
\]
Define the reference-powered tensors
\begin{equation}
\label{eq:ref_PQ}
\widetilde{\mathcal{P}}_i := \X_i\,\widetilde{\Xhat}_i^{\beta-2},
\qquad
\widetilde{\mathcal{Q}}_i := \widetilde{\Xhat}_i^{\beta-1}.
\end{equation}

\subsubsection{Transforms}
During the inner loop, we use the transforms (applied entrywise)
\[
\chi_{1,\beta}(Z,\widetilde Z)= \widetilde Z^{\,2-\beta}\odot Z^{\,\beta-1},
\qquad
\chi_{2,\beta}(Z,\widetilde Z)=
\begin{cases}
Z, & \beta<1,\\
Z^{\,\beta}\odot \widetilde Z^{-(\beta-1)}, & 1\le \beta<2.
\end{cases}
\]

Define the transformed core and factors:
\[
\G_{(1)} := \chi_{1,\beta}(\G,\widetilde{\G}),\quad
\G_{(2)} := \chi_{2,\beta}(\G,\widetilde{\G}),\quad
A^{(n)}_{(1)} := \chi_{1,\beta}(A^{(n)},\widetilde A^{(n)}),\quad
A^{(n)}_{(2)} := \chi_{2,\beta}(A^{(n)},\widetilde A^{(n)}).
\]

\subsubsection{Joint-MM core update, explicit formula}
Define the joint numerator and denominator tensors:
\begin{equation}
\label{eq:jmm_core_num}
\widetilde{\mathcal{P}}_{\mathrm{core,J}}(j_1,\dots,j_N)
=
\sum_{i_1=1}^{I_1}\cdots\sum_{i_N=1}^{I_N}
\widetilde{\mathcal{P}}_{i_1\dots i_N}
\prod_{n=1}^N A^{(n)}_{(1)}(i_n,j_n),
\end{equation}
\begin{equation}
\label{eq:jmm_core_den}
\widetilde{\mathcal{Q}}_{\mathrm{core,J}}(j_1,\dots,j_N)
=
\sum_{i_1=1}^{I_1}\cdots\sum_{i_N=1}^{I_N}
\widetilde{\mathcal{Q}}_{i_1\dots i_N}
\prod_{n=1}^N A^{(n)}_{(2)}(i_n,j_n).
\end{equation}
Then the inner joint-MM core update is
\[
\G \leftarrow \widetilde{\G} \od
\left(
\frac{\widetilde{\mathcal{P}}_{\mathrm{core,J}}}{\widetilde{\mathcal{Q}}_{\mathrm{core,J}}}
\right)^{\gamma(\beta)}.
\]

\subsubsection{Joint-MM factor update, explicit formula}
Fix a mode $n$.
Define the joint numerator and denominator matrices
$\mathrm{Num}^{(n)}_{\mathrm{J}}\in\R^{I_n\times J_n}$ and
$\mathrm{Den}^{(n)}_{\mathrm{J}}\in\R^{I_n\times J_n}$ by
\begin{equation}
\label{eq:jmm_factor_num}
\mathrm{Num}^{(n)}_{\mathrm{J}}(i_n,j_n)
=
\sum_{i_{-n}}
\sum_{j_1=1}^{J_1}\cdots\sum_{j_{n-1}=1}^{J_{n-1}}
\sum_{j_{n+1}=1}^{J_{n+1}}\cdots\sum_{j_N=1}^{J_N}
\widetilde{\mathcal{P}}_{i_1\dots i_N}\,
\G_{(1)}(j_1,\dots,j_N)
\prod_{m\neq n} A^{(m)}_{(1)}(i_m,j_m),
\end{equation}
\begin{equation}
\label{eq:jmm_factor_den}
\mathrm{Den}^{(n)}_{\mathrm{J}}(i_n,j_n)
=
\sum_{i_{-n}}
\sum_{j_1=1}^{J_1}\cdots\sum_{j_{n-1}=1}^{J_{n-1}}
\sum_{j_{n+1}=1}^{J_{n+1}}\cdots\sum_{j_N=1}^{J_N}
\widetilde{\mathcal{Q}}_{i_1\dots i_N}\,
\G_{(2)}(j_1,\dots,j_N)
\prod_{m\neq n} A^{(m)}_{(2)}(i_m,j_m).
\end{equation}
The inner joint-MM factor update is
\[
A^{(n)} \leftarrow \widetilde A^{(n)} \od
\left(
\frac{\mathrm{Num}^{(n)}_{\mathrm{J}}}{\mathrm{Den}^{(n)}_{\mathrm{J}}}
\right)^{\gamma(\beta)}.
\]

\begin{remark}
Equations \eqref{eq:jmm_factor_num} and \eqref{eq:jmm_factor_den} show that the factor update can be computed
directly by contracting $\widetilde{\mathcal{P}}$ or $\widetilde{\mathcal{Q}}$ with the transformed core and all
transformed factors except $A^{(n)}$.
This can be implemented by one einsum call per numerator and denominator, as shown in Appendix~\ref{app:einsum}.
\end{remark}

\section{Proofs for the KL-based convergence analysis of J-CoMM}
\label{app:jcomm_kl_proofs}

\subsection{Proof of Lemma~\ref{lem:jcomm_uniform_curvature}: uniform curvature of the scalar J-CoMM surrogates}
\label{app:proof_jcomm_uniform_curvature}

\begin{proof}
The explicit scalar forms follow from Lemma~\ref{lem:jmm_cp_scalar} for CP and from
Corollary~\ref{cor:jmm_tucker_scalar} for Tucker.

We first bound the admissible ratio variable.
By Assumption~\ref{ass:jcomm_conv}, every admissible scalar variable and every corresponding reference value satisfy
\[
\eps \le Z \le M,
\qquad
\eps \le \widetilde Z \le M.
\]
Hence
\[
\frac{\eps}{M}\le \frac{Z}{\widetilde Z}\le \frac{M}{\eps}.
\]
Therefore, for every scalar block update,
\[
u\in[\underline u,\overline u]
\qquad\text{with}\qquad
\underline u:=\frac{\eps}{M},
\quad
\overline u:=\frac{M}{\eps}.
\]

We next show that the contraction coefficients $\mathrm{Num}$ and $\mathrm{Den}$ are uniformly bounded above and below
by positive constants.
Fix one outer iterate $\widetilde\Theta\in\mathcal C$.
Since all blocks are entrywise bounded between $\eps$ and $M$, every model entry
$\widehat{\X}_i(\Theta)$ is a continuous positive function of $\Theta$ on the compact set $\mathcal C$.
Hence there exist constants
\[
0<\underline Xhat \le \widehat{\X}_i(\Theta)\le \overline Xhat<\infty
\qquad\text{for all } i \text{ and all } \Theta\in\mathcal C.
\]
In particular, for the reference reconstruction $\widetilde{\Xhat}=\widehat{\X}(\widetilde\Theta)$,
the reference-powered tensors
\[
\widetilde{\mathcal P}_i=\X_i\,\widetilde{\Xhat}_i^{\beta-2},
\qquad
\widetilde{\mathcal Q}_i=\widetilde{\Xhat}_i^{\beta-1}
\]
are continuous and strictly positive on $\mathcal C$.
Using Assumption~\ref{ass:jcomm_conv}(iv), we obtain positive bounds
\[
0<\underline P \le \widetilde{\mathcal P}_i \le \overline P <\infty,
\qquad
0<\underline Q \le \widetilde{\mathcal Q}_i \le \overline Q <\infty.
\]

Likewise, each transformed factor/core entry
\[
\chi_{1,\beta}(Z,\widetilde Z)=\widetilde Z^{\,2-\beta} Z^{\,\beta-1},
\qquad
\chi_{2,\beta}(Z,\widetilde Z)=
\begin{cases}
Z, & \beta<1,\\
Z^\beta \widetilde Z^{1-\beta}, & 1\le\beta<2,
\end{cases}
\]
is a continuous positive function on the compact box $[\eps,M]\times[\eps,M]$.
Therefore there exist constants
\[
0<\underline \chi_1 \le \chi_{1,\beta}(Z,\widetilde Z)\le \overline \chi_1<\infty,
\qquad
0<\underline \chi_2 \le \chi_{2,\beta}(Z,\widetilde Z)\le \overline \chi_2<\infty
\]
uniformly over all admissible pairs $(Z,\widetilde Z)$.

Now, for both CP and Tucker, the coefficients $\mathrm{Num}$ and $\mathrm{Den}$ are finite sums of products of
the form
\[
\widetilde{\mathcal P}_i \times \prod \chi_{1,\beta}(\cdot,\widetilde{\cdot}),
\qquad
\widetilde{\mathcal Q}_i \times \prod \chi_{2,\beta}(\cdot,\widetilde{\cdot}),
\]
respectively.
Each summand is continuous and strictly positive on the compact feasible set.
Since the number of summands is finite, it follows that there exist positive constants
\[
0<\underline N\le \mathrm{Num}\le \overline N<\infty,
\qquad
0<\underline D\le \mathrm{Den}\le \overline D<\infty,
\]
uniformly for all admissible scalar block subproblems.

We now lower-bound the second derivative of the scalar surrogate.

For $1<\beta<2$, we have
\[
g''(u)
=
\mathrm{Den}(\beta-1)u^{\beta-2}
+
\mathrm{Num}(2-\beta)u^{\beta-3}.
\]
Since $\beta-2<0$ and $\beta-3<0$, the functions $u^{\beta-2}$ and $u^{\beta-3}$ are decreasing on
$(0,\infty)$.
Hence, for all $u\in[\underline u,\overline u]$,
\[
g''(u)
\ge
\underline D(\beta-1)\overline u^{\,\beta-2}
+
\underline N(2-\beta)\overline u^{\,\beta-3}
=: \mu_\beta >0.
\]

For $\beta=1$, we have
\[
g''(u)=\frac{\mathrm{Num}}{u^2}.
\]
Therefore,
\[
g''(u)\ge \frac{\underline N}{\overline u^2}=: \mu_1>0.
\]

For $0\le\beta<1$, we have
\[
g''(u)
=
\mathrm{Num}(2-\beta)u^{\beta-3},
\]
and again $\beta-3<0$, so
\[
g''(u)\ge \underline N(2-\beta)\overline u^{\,\beta-3}=: \mu_\beta>0.
\]

Thus, in every case, there exists a constant $\mu>0$ such that
\[
g''(u)\ge \mu
\qquad \text{for all admissible }u.
\]
Hence each scalar surrogate is $\mu$-strongly convex on $[\underline u,\overline u]$.
Strong convexity on a closed interval implies uniqueness of the minimizer $u^\star$ and yields
\[
g(u)-g(u^\star)\ge \frac{\mu}{2}|u-u^\star|^2,
\]
which concludes the proof.
\end{proof}

\subsection{Proof of Lemma~\ref{lem:jcomm_sufficient_decrease}: sufficient decrease for one outer J-CoMM step}
\label{app:proof_jcomm_sufficient_decrease}
\begin{proof}
Fix one outer iteration $k$ and let the reference point be
\[
\widetilde\Theta=\Theta^k.
\]
During this outer iteration, the surrogate
\[
G(\cdot\mid \Theta^k)
\]
is kept fixed, and the algorithm performs one cyclic sweep of exact block minimizations.

For each block index $b\in\{1,\dots,B\}$, the transition
\[
\Theta^{k,b-1}\longmapsto \Theta^{k,b}
\]
updates only block $b$, all other blocks being fixed.
By construction of J-CoMM with $L=1$, block $b$ has not been updated earlier in the current sweep.
Hence, immediately before its update, block $b$ is still equal to its reference value in $\Theta^k$.
Therefore, in the scalar ratio parametrization
\[
u=\frac{Z}{\widetilde Z},
\]
the pre-update value is
\[
u_{\mathrm{old}}=1.
\]

Now consider any scalar entry of the active block.
By Lemma~\ref{lem:jcomm_uniform_curvature}, the corresponding scalar surrogate
$g(u)$ is uniformly $\mu$-strongly convex on its admissible interval, where $\mu>0$ is independent of $k$ and of the
chosen scalar subproblem.
Let $u^\star$ denote its minimizer.
Since the J-CoMM block update is the exact minimizer of the scalar surrogate, the updated scalar value satisfies
\[
u_{\mathrm{new}}=u^\star.
\]
Strong convexity then yields
\[
g(1)-g(u^\star)\ge \frac{\mu}{2}|1-u^\star|^2.
\]

We now translate this estimate back to the original variable.
Since
\[
u^\star-1=\frac{Z^\star-\widetilde Z}{\widetilde Z},
\]
and since $\widetilde Z\le M$ by Assumption~\ref{ass:jcomm_conv}, we have
\[
|u^\star-1|
=
\frac{|Z^\star-\widetilde Z|}{\widetilde Z}
\ge
\frac{|Z^\star-\widetilde Z|}{M}.
\]
Therefore,
\[
g(1)-g(u^\star)
\ge
\frac{\mu}{2M^2}|Z^\star-\widetilde Z|^2.
\]

Summing this inequality over all scalar entries of the active block shows that the decrease of the fixed surrogate
during block update $b$ satisfies
\[
G(\Theta^{k,b-1}\mid \Theta^k)-G(\Theta^{k,b}\mid \Theta^k)
\ge
\frac{\mu}{2M^2}\,\|\Theta^{k,b}-\Theta^{k,b-1}\|^2.
\]
Define
\[
c_{\mathrm{dec}}:=\frac{\mu}{2M^2}>0.
\]
Summing over the $B$ block updates in the sweep yields
\[
G(\Theta^{k,0}\mid \Theta^k)-G(\Theta^{k,B}\mid \Theta^k)
\ge
c_{\mathrm{dec}}\sum_{b=1}^B \|\Theta^{k,b}-\Theta^{k,b-1}\|^2.
\]
Since $\Theta^{k,0}=\Theta^k$ and $\Theta^{k,B}=\Theta^{k+1}$, this becomes
\[
G(\Theta^k\mid \Theta^k)-G(\Theta^{k+1}\mid \Theta^k)
\ge
c_{\mathrm{dec}}\sum_{b=1}^B \|\Theta^{k,b}-\Theta^{k,b-1}\|^2.
\]

We now pass from the surrogate decrease to objective decrease.
By tightness of the surrogate at the reference iterate,
\[
F(\Theta^k)=G(\Theta^k\mid \Theta^k),
\]
and by majorization,
\[
F(\Theta^{k+1})\le G(\Theta^{k+1}\mid \Theta^k).
\]
Hence
\[
F(\Theta^k)-F(\Theta^{k+1})
\ge
G(\Theta^k\mid \Theta^k)-G(\Theta^{k+1}\mid \Theta^k)
\ge
c_{\mathrm{dec}}\sum_{b=1}^B \|\Theta^{k,b}-\Theta^{k,b-1}\|^2.
\]

Finally, the block increments have disjoint support in the product space of all variables, since at each stage
only one block is modified.
Therefore,
\[
\Theta^{k+1}-\Theta^k
=
\sum_{b=1}^B \bigl(\Theta^{k,b}-\Theta^{k,b-1}\bigr),
\]
and the orthogonality of block supports gives
\[
\|\Theta^{k+1}-\Theta^k\|^2
=
\sum_{b=1}^B \|\Theta^{k,b}-\Theta^{k,b-1}\|^2.
\]
Combining the previous estimates proves
\[
F(\Theta^k)-F(\Theta^{k+1})
\ge
c_{\mathrm{dec}}\|\Theta^{k+1}-\Theta^k\|^2.
\]

Since all iterates belong to the feasible set $\mathcal C$, we also have
\[
\Psi(\Theta^k)=F(\Theta^k),\qquad \Psi(\Theta^{k+1})=F(\Theta^{k+1}),
\]
so the same inequality holds for $\Psi$.
\end{proof}

\subsection{Proof of Lemma~\ref{lem:jcomm_relative_error}: relative-error bound}
\label{app:proof_jcomm_relative_error}

\begin{proof}
Fix one outer iteration $k$ and denote the reference point by
\[
\widetilde\Theta:=\Theta^k.
\]
As in Lemma~\ref{lem:jcomm_sufficient_decrease}, let
\[
\Theta^{k,0}:=\Theta^k,\qquad \Theta^{k,B}:=\Theta^{k+1},
\]
and, for each block index $b\in\{1,\dots,B\}$, let $\Theta^{k,b}$ denote the intermediate iterate obtained
after updating the first $b$ blocks of the fixed surrogate $G(\cdot\mid \Theta^k)$.

For each block $b$, the J-CoMM update computes an exact minimizer of the block subproblem
\[
\min_{Z_b\in \mathcal C_b}
\; G(\Theta^{k,b-1}_{<b},\, Z_b,\, \Theta^k_{>b}\mid \Theta^k),
\]
where $\Theta^{k,b-1}_{<b}$ denotes the already updated blocks and $\Theta^k_{>b}$ the blocks not yet updated
in the current sweep.
Since $\mathcal C_b$ is closed and convex and $G(\cdot\mid \Theta^k)$ is differentiable with respect to its first
argument, the first-order optimality condition gives
\[
0 \in \nabla_b G(\Theta^{k,b}\mid \Theta^k) + N_{\mathcal C_b}(\Theta_b^{k,b}),
\]
where $N_{\mathcal C_b}$ denotes the normal cone to $\mathcal C_b$.
Therefore, for each $b$, there exists a vector
\[
n_b^{k+1}\in N_{\mathcal C_b}(\Theta_b^{k+1})
\]
such that
\[
\nabla_b G(\Theta^{k,b}\mid \Theta^k) + n_b^{k+1}=0.
\]
Observe that this is well-defined because, once block $b$ has been updated, it is never modified again during the
same sweep, so
\[
\Theta_b^{k,b}=\Theta_b^{k+1}.
\]

Now define the block residual
\[
\xi_b^{k+1}:=\nabla_b F(\Theta^{k+1}) + n_b^{k+1}.
\]
Since $\mathcal C=\mathcal C_1\times\cdots\times \mathcal C_B$, the product normal-cone formula yields
\[
n^{k+1}:=(n_1^{k+1},\dots,n_B^{k+1})\in N_{\mathcal C}(\Theta^{k+1}),
\]
and therefore
\[
\xi^{k+1}:=(\xi_1^{k+1},\dots,\xi_B^{k+1})
=
\nabla F(\Theta^{k+1}) + n^{k+1}
\in \partial \Psi(\Theta^{k+1}).
\]
Hence
\[
\operatorname{dist}\bigl(0,\partial \Psi(\Theta^{k+1})\bigr)
\le \|\xi^{k+1}\|.
\]

It remains to bound $\|\xi^{k+1}\|$ by the step length.
Using the optimality relation above, we obtain for each block $b$
\[
\xi_b^{k+1}
=
\nabla_b F(\Theta^{k+1}) - \nabla_b G(\Theta^{k,b}\mid \Theta^k).
\]
Add and subtract $\nabla_b F(\Theta^k)$ and $\nabla_b G(\Theta^k\mid \Theta^k)$:
\[
\xi_b^{k+1}
=
\bigl(\nabla_b F(\Theta^{k+1})-\nabla_b F(\Theta^k)\bigr)
+
\bigl(\nabla_b F(\Theta^k)-\nabla_b G(\Theta^k\mid \Theta^k)\bigr)
+
\bigl(\nabla_b G(\Theta^k\mid \Theta^k)-\nabla_b G(\Theta^{k,b}\mid \Theta^k)\bigr).
\]

The middle term vanishes by Assumption~\ref{ass:jcomm_conv}(vii), namely the first-order consistency of the surrogate
at the reference point:
\[
\nabla_1 G(\Theta^k\mid \Theta^k)=\nabla F(\Theta^k).
\]
Thus
\[
\xi_b^{k+1}
=
\bigl(\nabla_b F(\Theta^{k+1})-\nabla_b F(\Theta^k)\bigr)
+
\bigl(\nabla_b G(\Theta^k\mid \Theta^k)-\nabla_b G(\Theta^{k,b}\mid \Theta^k)\bigr).
\]

Using the Lipschitz continuity of $\nabla F$ and of $\nabla_1 G(\cdot\mid \widetilde\Theta)$ on $\mathcal C$,
we obtain
\[
\|\xi_b^{k+1}\|
\le
L_F \|\Theta^{k+1}-\Theta^k\|
+
L_G \|\Theta^{k,b}-\Theta^k\|.
\]
Since $\Theta^{k,b}-\Theta^k$ contains only the first $b$ block increments of the sweep, while
$\Theta^{k+1}-\Theta^k$ contains all of them, the disjoint-support structure of block increments gives
\[
\|\Theta^{k,b}-\Theta^k\|
\le
\|\Theta^{k+1}-\Theta^k\|.
\]
Therefore
\[
\|\xi_b^{k+1}\|
\le
(L_F+L_G)\|\Theta^{k+1}-\Theta^k\|.
\]

Finally, summing over the $B$ blocks and using the product norm yields
\[
\|\xi^{k+1}\|^2
=
\sum_{b=1}^B \|\xi_b^{k+1}\|^2
\le
B(L_F+L_G)^2 \|\Theta^{k+1}-\Theta^k\|^2.
\]
Hence
\[
\|\xi^{k+1}\|
\le
\sqrt{B}\,(L_F+L_G)\,\|\Theta^{k+1}-\Theta^k\|.
\]
Since $\xi^{k+1}\in\partial \Psi(\Theta^{k+1})$, this proves
\[
\operatorname{dist}\bigl(0,\partial \Psi(\Theta^{k+1})\bigr)
\le
c_{\mathrm{err}}\,\|\Theta^{k+1}-\Theta^k\|,
\qquad
c_{\mathrm{err}}:=\sqrt{B}\,(L_F+L_G).
\]
\end{proof}

\subsection{Proof of Proposition~\ref{prop:jcomm_cluster}: asymptotic regularity and critical cluster points}
\label{app:proof_jcomm_cluster}
\begin{proof}
By Lemma~\ref{lem:jcomm_sufficient_decrease},
\[
\Psi(\Theta^k)-\Psi(\Theta^{k+1})
\ge
c_{\mathrm{dec}}\|\Theta^{k+1}-\Theta^k\|^2
\qquad\text{for all }k.
\]
Hence $\{\Psi(\Theta^k)\}$ is nonincreasing.
Since $\mathcal C$ is compact and $F$ is continuous on $\mathcal C$, the function
\[
\Psi=F+\iota_{\mathcal C}
\]
is bounded below on the generated sequence.
Therefore, $\{\Psi(\Theta^k)\}$ converges to some finite limit $\Psi_\infty$.

Summing the sufficient decrease inequality from $k=0$ to $K$ yields
\[
c_{\mathrm{dec}}\sum_{k=0}^K \|\Theta^{k+1}-\Theta^k\|^2
\le
\Psi(\Theta^0)-\Psi(\Theta^{K+1})
\le
\Psi(\Theta^0)-\Psi_\infty.
\]
Letting $K\to\infty$ gives
\[
\sum_{k=0}^\infty \|\Theta^{k+1}-\Theta^k\|^2 < \infty.
\]
In particular,
\[
\|\Theta^{k+1}-\Theta^k\|\to 0.
\]

Now let $\Theta^\star$ be any cluster point of $\{\Theta^k\}$.
By compactness of $\mathcal C$, there exists a subsequence $\Theta^{k_j}\to\Theta^\star$.
Since $\|\Theta^{k+1}-\Theta^k\|\to 0$, we also have
\[
\Theta^{k_j+1}\to \Theta^\star.
\]
Applying Lemma~\ref{lem:jcomm_relative_error}, we obtain
\[
\operatorname{dist}\bigl(0,\partial \Psi(\Theta^{k_j+1})\bigr)
\le
c_{\mathrm{err}}\|\Theta^{k_j+1}-\Theta^{k_j}\|\to 0.
\]
Thus there exist vectors $\xi^{k_j+1}\in \partial \Psi(\Theta^{k_j+1})$ such that
\[
\|\xi^{k_j+1}\|\to 0.
\]

Since $\mathcal C$ is closed and $F$ is continuous on $\mathcal C$, the function
$\Psi=F+\iota_{\mathcal C}$ is proper and lower semicontinuous.
Since $\Theta^{k_j+1}\to\Theta^\star$ and $\xi^{k_j+1}\to 0$, the closedness of the limiting subdifferential implies
\[
0\in \partial \Psi(\Theta^\star).
\]
Hence every cluster point is a critical point of $\Psi$.
\end{proof}

\section{Practical einsum recipes}
\label{app:einsum}

This appendix translates the main contractions into explicit einsum calls.
We write third-order examples because they are easiest to read.
Higher-order cases follow the same index logic.

\subsection{CP contractions (third-order)}
Let $\mathcal{T}\in\R^{I\times J\times K}$, $B^{(2)}\in\R^{J\times R}$, $B^{(3)}\in\R^{K\times R}$.
\begin{lstlisting}
# CPContr^(1)(T; B2, B3) in R^{I x R}
M = einsum('ijk,jr,kr->ir', T, B2, B3, optimize=True)
\end{lstlisting}
Similarly,
\begin{lstlisting}
# CPContr^(2)(T; B1, B3) in R^{J x R}
M = einsum('ijk,ir,kr->jr', T, B1, B3, optimize=True)

# CPContr^(3)(T; B1, B2) in R^{K x R}
M = einsum('ijk,ir,jr->kr', T, B1, B2, optimize=True)
\end{lstlisting}

\subsection{Tucker reconstruction and block-MM quantities (third-order)}
Let $\G\in\R^{A\times B\times C}$,
$A^{(1)}\in\R^{I\times A}$, $A^{(2)}\in\R^{J\times B}$, $A^{(3)}\in\R^{K\times C}$.
\begin{lstlisting}
# Reconstruction Xhat in R^{I x J x K}
Xhat = einsum('abc,ia,jb,kc->ijk', G, A1, A2, A3, optimize=True)

# Core contraction P_core in R^{A x B x C}
P_core = einsum('ijk,ia,jb,kc->abc', P, A1, A2, A3, optimize=True)

# Partial tensor B^(1) in R^{A x J x K}
B1 = einsum('abc,jb,kc->ajk', G, A2, A3, optimize=True)

# Factor numerator Num^(1) in R^{I x A}
Num1 = einsum('ijk,ajk->ia', P, B1, optimize=True)
\end{lstlisting}
One can avoid forming $B^{(1)}$ and compute $\mathrm{Num}^{(1)}$ directly:
\begin{lstlisting}
Num1 = einsum('ijk,abc,jb,kc->ia', P, G, A2, A3, optimize=True)
Den1 = einsum('ijk,abc,jb,kc->ia', Q, G, A2, A3, optimize=True)
\end{lstlisting}
This direct form is often preferable because it reduces memory traffic.

\subsection{Tucker joint-MM quantities (third-order)}
Assume reference-powered tensors $\widetilde P$ and $\widetilde Q$.
Assume transformed blocks G1, G2, A1\_1, A2\_1, A3\_1 and A1\_2, A2\_2, A3\_2.

\paragraph{Core update.} \text{   }
\begin{lstlisting}
Num_core = einsum('ijk,ia,jb,kc->abc', Ptilde, A1_1, A2_1, A3_1, optimize=True)
Den_core = einsum('ijk,ia,jb,kc->abc', Qtilde, A1_2, A2_2, A3_2, optimize=True)
G = Gtilde * (Num_core / Den_core)**gamma
\end{lstlisting}

\paragraph{Factor update for mode 1.}
This computes Num\_J\^(1) and Den\_J\^(1) in R\^\{I x A\}:
\begin{lstlisting}
Num1 = einsum('ijk,abc,jb,kc->ia', Ptilde, G1, A2_1, A3_1, optimize=True)
Den1 = einsum('ijk,abc,jb,kc->ia', Qtilde, G2, A2_2, A3_2, optimize=True)
A1 = A1tilde * (Num1 / Den1)**gamma
\end{lstlisting}
The same pattern holds for modes 2 and 3, by permuting indices accordingly.

\subsection{Practical remarks for implementation}
\begin{itemize}
\item Use \texttt{optimize=True} in einsum to let the backend choose a good contraction path.
\item Cache the reference tensors $\widetilde{\Xhat}$, $\widetilde{\mathcal{P}}$, and $\widetilde{\mathcal{Q}}$ for the whole inner loop.
\item For large tensors, the dominant cost is often memory traffic.
Direct einsum contractions that avoid materializing intermediate tensors can be faster.
\end{itemize}

\bibliographystyle{siamplain}
\bibliography{references}

\end{document}